\documentclass[10pt]{amsart}
\usepackage{bbm}
\usepackage{cases}
\usepackage{txfonts}
\usepackage{amsfonts}
\usepackage{mathrsfs}
\usepackage{amssymb}
\usepackage{amsmath}
\usepackage{epic}

\newtheorem{proposition}{Proposition}[section]
\newtheorem{lemma}[proposition]{Lemma}
\newtheorem{corollary}[proposition]{Corollary}
\newtheorem{theorem}[proposition]{Theorem}

\theoremstyle{definition}
\newtheorem{definition}[proposition]{Definition}

\theoremstyle{remark}
\newtheorem{remark}[proposition]{Remark}

\newcommand{\thlabel}[1]{\label{th:#1}}
\newcommand{\thref}[1]{Theorem~\ref{th:#1}}
\newcommand{\selabel}[1]{\label{se:#1}}
\newcommand{\seref}[1]{Section~\ref{se:#1}}
\newcommand{\lelabel}[1]{\label{le:#1}}
\newcommand{\leref}[1]{Lemma~\ref{le:#1}}
\newcommand{\prlabel}[1]{\label{pr:#1}}
\newcommand{\prref}[1]{Proposition~\ref{pr:#1}}
\newcommand{\colabel}[1]{\label{co:#1}}
\newcommand{\coref}[1]{Corollary~\ref{co:#1}}
\newcommand{\relabel}[1]{\label{re:#1}}

\newcommand{\delabel}[1]{\label{de:#1}}
\newcommand{\deref}[1]{Definition~\ref{de:#1}}

\def\a{\alpha}
\def\b{\beta}

\def\d{\delta}
\def\D{\Delta}

\def\ep{\varepsilon}

\def\g{\gamma}

\def\l{\lambda}

\def\ol{\overline}
\def\op{\oplus}
\def\ot{\otimes}

\def\oo{\infty}

\def\ra{\rightarrow}
\def\s{\sigma}

\def\ss{\subseteq}

\def\ti{\times}

\def\<{\leqslant}
\def\>{\geqslant}

\date{}
\begin{document}
\title[Representations of Hopf-Ore extensions of group algebras]{Representations of Hopf-Ore extensions of group algebras and pointed
Hopf algebras of rank one}

\author{Zhen Wang}
\address{Department of Fundamental Sciences, Yancheng Institute
of Technology, Yancheng 224051, China\\
Department of Mathematics, Shanghai Jiaotong University, Shanghai 200240, China}
\email{wangzhen118@gmail.com}

\author{Lan You}
\address{Department of Fundamental Sciences, Yancheng Institute
of Technology, Yancheng 224051, China}
\email{youlan0125@sina.com}

\author[Hui-Xiang Chen]{Hui-Xiang Chen$^*$}
\thanks{$^*$Corresponding author}
\address{School of Mathematical Science, Yangzhou University,
Yangzhou 225002, China}
\email{hxchen@yzu.edu.cn}


\begin{abstract}
In this paper, we study the representation theory of Hopf-Ore extensions of group algebras
and pointed Hopf algebras of rank one over an arbitrary field $k$.
Let $H=kG(\chi, a,\d)$ be a Hopf-Ore extension of $kG$ and $H'$ a rank one quotient Hopf algebra of $H$,
where $k$ is a field, $G$ is a group, $a$ is a central
element of $G$ and $\chi$ is a $k$-valued character for $G$ with $\chi(a)\neq 1$.
We first show that the simple weight modules over $H$ and $H'$ are finite dimensional.
Then we describe the structures of all simple weight modules over $H$ and $H'$, and classify them.
We also consider the decomposition of the tensor product of two simple weight modules over $H'$
into the direct sum of indecomposable modules.
Furthermore, we describe the structures of finite dimensional indecomposable weight modules over
$H$ and $H'$, and classify them. Finally, when $\chi(a)$ is a primitive $n$-th root of unity for some $n\>2$,
we determine all finite dimensional indecomposable projective objects in the category of weight modules
over $H'$.
\end{abstract}
\maketitle

\section*{Introduction}
The construction and classification of Hopf algebras play an
important role in the theory of Hopf algebras. During the last few
years several classification results for pointed Hopf algebras were
obtained based on the theory of Nichols algebras \cite{AndSch98,
AndSch02, AndSch10}. Beattie, D$\check{a}$sc$\check{a}$lescu and Gr$\ddot{u}$nenfelder
constructed a large class of pointed Hopf algebras by Ore extension in \cite{BeaDasGrun} .
Panov studied Ore extensions in the class of Hopf algebras,
which enable one to describe the Hopf-Ore extensions for the group algebras,
for the algebras $U({\mathfrak g})$ and $U_q({\mathfrak g})$, and for the
quantum $``ax+b"$ group in \cite{Pa}. Krop and Radford defined the rank as a measure
of complexity for Hopf algebras \cite{KR}. They classified
all finite dimensional pointed Hopf algebras of rank one over an algebraically
field $k$ of characteristic $0$.
Scherotzke classified such Hopf algebras
for the case of char$(k)=p>0$ in \cite{Sc}. It was shown in \cite{KR, Sc}
that a finite dimensional pointed Hopf algebra over an algebraically closed field
is isomorphic to a quotient of a Hopf-Ore extension of its coradical.
On the other hand, the representation theory of pointed Hopf algebras got many achievements
during past several years \cite{AndBea, Chen00, Chen02, Cib, Igl09, Lorenz, ZhChen}.
One of the main tools in the abelian coradical case is the consideration of
Verma type modules and the corresponding highest weight modules. Using a reductive program,
Andruskiewitsch, Radford and Schneider \cite{AndRadSch} investigated the representation theory
of a large class of pointed Hopf algebras in the abelian case, extending results
of Lusztig and others. For the non abelian case, some results were obtained by Iglesias
\cite{Igl09}, and recently by Pogorelsky and Vay \cite{PoVay}.

In this paper, we study the representation theory of Hopf-Ore extensions of group algebras
and pointed Hopf algebras of rank one.
The paper is organized as follows. In \seref{1}, we present
some basic definitions and facts about
Hopf-Ore extensions.
In \seref{2}, for later use, we summarize some results about the Hopf-Ore extensions of group algebras
and finite dimensional pointed Hopf algebras of rank one from the references \cite{KR, Pa, Sc}.
By a little modification of the arguments of \cite{KR, Sc}, we generalize these results
to the general case without any restrictions for the base field and the dimensions of Hopf algebras.
We compute the rank of Hopf-Ore extension $kG[x;\tau, \d]=kG(\chi,a,\d)$
of any group algebra $kG$, where $\tau$ is an automorphism of $kG$,
$\d$ is a $\tau$-derivation of $kG$, $\chi$ is a $k$-valued character for $G$
and $a\in Z(G)$, the center of $G$. It is shown that any (finite or infinite dimensional) pointed Hopf algebra
of rank one over an arbitrary field $k$ is isomorphic to a quotient of
Hopf-Ore extension of its coradical.
In \seref{3}, we study the representation theory of the Hopf algebra
$H=kG(\chi, a,\d)$, the Hopf-Ore extension of $kG$, and its quotient Hopf algebra
$H'=H/I$ of rank one, where $\chi(a)\neq 1$.
We show that any simple weight module over $H$ (or $H'$) is finite dimensional.
Then all simple weight modules over $H$ (or $H'$) are described and classified.
Finally, we describe the decomposition of the tensor product of two
simple weight modules over $H'$ into a direct sum of indecomposable modules.
In \seref{4}, we continue the study of \seref{3}, and investigate the finite dimensional indecomposable
weight modules over $H$ (or $H'$). We describe the structures of all finite dimensional indecomposable
weight modules over $H$ (or $H'$), and classify them. We also give
a class of infinite dimensional indecomposable projective objects in the category
of weight modules over $H$. Finally, when $\chi(a)$ is a primitive $n$-th root of unity for some $n\>2$,
we determine all finite dimensional indecomposable projective objects in the category of weight modules
over $H'$.

\section{Preliminaries}\selabel{1}

Throughout, we work over a field $k$. Unless otherwise stated, all algebras, coalgebras,
Hopf algebras and modules are vector spaces over $k$. All linear maps mean $k$-linear maps,
dim and $\otimes$ mean dim$_k$ and $\otimes_k$, respectively.
Our reference for basic concepts and notations about Hopf algebras is \cite{Mo}.
In particular, for a Hopf algebra, we will use $\ep$, $\D$ and $S$ to denote the counit,
comultiplication and antipode, respectively. We use Sweedler's notations for the comultiplication
and coaction. Let $k^{\ti}=k\setminus\{0\}$.
For a group $G$, let $\widehat{G}$ denote the set of characters of $G$ over $k$,
and let $Z(G)$ denote the center of $G$.

Let us recall the concept of the rank of a Hopf algebra defined in \cite{KR}.

\begin{definition}\delabel{1.1}
Let $H$ be a Hopf algebra, $H_0\subseteq H_1\subseteq H_2\subseteq \cdots$
the coradical filtration of $H$.
Assume that $H_0$ is a Hopf subalgebra. If $H$ is generated by $H_1$ as an algebra and
${\rm dim}_k(k\ot _{H_0} H_1)=n+1$, then we say that $H$ is a Hopf algebra of rank $n$.
\end{definition}

Note that each $H_i$ is a free $H_0$-module in the case of \deref{1.1}, $i\>0$.

A coalgebra $C$ is {\it graded} if there exist subspaces $\{C_{(n)}\}_{n\>0}$
of $C$ such that $C=\oplus_{n\>0}C_{(n)}$
and $\D (C_{(n)})\ss\sum_{i=0}^n C_{(i)}\ot C_{(n-i)}$ for all
$n\>0$ and $\ep(C_{(n)})=0$ for all $n>0$.
If this is the case, then $C\ot C$ is also a graded coalgebra with grading
$(C\ot C)_{(n)}=\sum_{i=0}^n C_{(i)}\ot C_{(n-i)}$, $n\>0$. Moreover,
the comultiplication $\D: C\ra C\ot C$ of $C$ is a graded map.

For a coalgebra $C$, let $G(C)$ denote the set of all group-like elements
in $C$. For $g, h\in G(C)$, $c\in C$ is
called $(g,h)$-{\it primitive} if $\D(c)=c\ot g+h\ot c$. The set of all
$(g,h)$-primitives is denoted by $P_{g,h}(C)$, which is a subspace of $C$.
If $H$ is a Hopf algebra, then $G(H)$ is a group.

Let $A$ be a $k$-algebra. Let $\tau $ be an algebra endomorphism of
$A$ and $\d$ a $\tau$-derivation of $A$.
The {\it Ore extension} $A[y;\tau,\d]$ of the algebra $A$ is an algebra
generated by the variable $y$ and the algebra
$A$ with the relation $ya=\tau (a)y+\d(a)$ for all $a\in A$.
$A[y;\tau,\d]$ is a free left $A$-module with the $A$-basis $\{y^j\mid j\>0\}$
(see \cite{kas, Pa}).

\begin{definition}\delabel{1.2}(\cite[Definition 1.0]{Pa})
Assume that $A$ and $A[y; \tau, \d]$ are Hopf algebras. The Hopf algebra
$A[y; \tau, \d]$ is called a Hopf-Ore extension of $A$ if
$\D(y)=y\ot r_1+r_2\ot y$ for some $r_1$, $r_2\in A$ and
$A$ is a Hopf subalgebra of $A[y; \tau, \d]$.
\end{definition}

\begin{remark}\relabel{1.3}
In \deref{1.2}, both $r_1$ and $r_2$ are group-like elements \cite{Pa}.
Replacing the generating element $y$ by $r_2^{-1}y$, one may assume that
the element $y$ satisfies the relation
\begin{equation*}
    \D(y)=y\ot r+1\ot y
\end{equation*}
for some $r\in G(A)$. Here we make some
modification for $\D(y)$, which is assumed to be $ \D(y)=y\ot 1 +r\ot y$ in \cite{Pa}.
Hence there are also some corresponding changes in the following \leref{1.4} and related results
from the original ones in \cite{Pa}. Note that $\ep(y)=0$ and
$S(y)=-yr^{-1}$ in this case, where $S$ is the antipode of $A[y;\tau,\d]$.
\end{remark}

\begin{lemma}\lelabel{1.4}(\cite[Theorem 1.3]{Pa})
Assume that $A$ and $A[y;\tau,\d]$ are Hopf algebras. Then
$A[y;\tau,\d]$ is a Hopf-Ore extension of $A$ if and only if
the following conditions are satisfied:\\
{\rm (a)} there is a character $\chi:A\ra k$ such that
$\tau (a)=\sum a_1\chi(a_2)$ for all $a\in A$;\\
{\rm (b)} there is an $r\in G(A)$ such that $\sum a_1\chi(a_2)=\sum \chi(a_1) ad_r(a_2)$ for all $a\in A$,
where $ad_r(a)=rar^{-1}$;\\
{\rm (c)} $\D \d (a)=\sum \d(a_1)\ot ra_2+\sum a_1\ot \d(a_2)$
for all $a\in A$, where $r$ is given in (b).
\end{lemma}

Let $\d\in{\rm End}_k(A)$. $\d$ is called an $r$-{\it coderivation} if
$\D\d(a)=\sum \d(a_1)\ot ra_2+\sum a_1\ot \d(a_2)$, $a\in A$.
$\d$ is said to be a $\langle\chi, r\rangle$-{\it derivation}
if $\d$ is a $\tau $-derivation and an $r$-coderivation,
where $\chi$ is a character of $A$ determined by $\tau$ as in \leref{1.4}(a).
Denote the Hopf-Ore extension $A[y;\tau,\d]$ by $A(\chi,r,\d)$,
where $\chi:A\ra k$ is a character,
$r$ is a group-like element of $A$ as in \leref{1.4}, and $\d$ is a
$\langle \chi, r \rangle$-derivation \cite{Pa}.

Let $G$ be a group, $\chi\in\widehat{G}$. A linear map $\a: kG\ra k$ is called a $1$-{\it cocycle}
of $G$ with respect to $\chi$ if
$\a(gh)=\a(g)+\chi (g)\a(h)$ for all $g$, $h\in G$.
Let $Z_{\chi}^1(G)$ denote the set of all $1$-cocycles of $G$ with respect to $\chi$.

\begin{lemma}\lelabel{1.5}(\cite[Proposition 2.2]{Pa})
Every Hopf-Ore extension of $kG$ is of the form $kG(\chi,r,\d)$,
where $\chi\in\widehat{G}$, $r\in Z(G)$, and $\d(a)=\sum a_1(1-r)\a(a_2)$
for some $\a\in Z_{\chi}^1(G)$.
\end{lemma}

For the definition of $q$-binomial coefficients
$\binom{n}{m}_q$, the reader is directed to \cite[Page 74]{kas}.

\begin{lemma}\lelabel{1.6}
Assume that $B$ is a bialgebra. Let $a\in G(B)$ and $x\in B$ with $\D(x)=x\ot a+1\ot x$.
Let $q\in k^{\times}$ and $\b\in k$.\\
{\rm (a) } If $xa=qax+\b a(1-a)$,
then for any $n>0$ we have
\begin{equation*}
\D(x^n)=\sum_{l=0}^n\binom{n}{l}_qx^{n-l}\ot a^{n-l}x^l+P_n,
\quad P_n\in \sum_{i+j\<n-1}(B(i)\ot B(j)),
\end{equation*}
where $B(i)={\rm span}\{a^lx^i\mid l\>0\}$ for $i\>0$.\\
{\rm (b) }  Assume that char$(k)=p>0$. If $xa=ax+a(1-a)$, then for all
$r\>1$,
\begin{equation*}
\D(x^{p^r}-x^{p^{r-1}})=(x^{p^r}-x^{p^{r-1}})\ot a^{p^r}+1\ot
(x^{p^r}-x^{p^{r-1}}).
\end{equation*}
\end{lemma}
\begin{proof}
Part (a) can be shown by the induction on $n$ using \cite[Proposition IV.2.1(c)]{kas}. For Part
(b), by \cite[Corollary 4.10]{Sc}, we have
$\D(x^p)=x^p\ot a^p+(p-1)!x\ot a^p+x\ot a+1\ot x^p$. Hence
$\D(x^p-x)=x^p\ot a^p+(p-1)!x\ot a^p+x\ot a+1\ot
x^p-x\ot a-1\ot x
=(x^p-x)\ot a^p+1\ot (x^p-x)$.
Note that $xa^p=a^px$, and so $((x^p-x)\ot a^p)(1\ot (x^p-x))=(1\ot
(x^p-x))((x^p-x)\ot a^p)$. Thus, for $r\>2$, we have
$\D(x^{p^r}-x^{p^{r-1}})=[\D(x^p-x)]^{p^{r-1}}=(x^{p^r}-x^{p^{r-1}})\ot a^{p^r}+1\ot (x^{p^r}-x^{p^{r-1}})$.
\end{proof}

\section{Classifications}\selabel{2}

In this section, we investigate the ranks of Hopf-Ore extensions of a group algebra
and the classification of pointed Hopf algebras of rank one by following the ideas
of \cite{KR, Ra, Sc}.

\subsection{The ranks of Hopf-Ore extensions of a group algebra}
In this subsection, assume that $H=kG[x; \tau, \d]$ is a Hopf-Ore
extension of $kG$, where $G$ is a group. By \leref{1.5}, $H=kG(\chi, a, \d)$,
where $\chi \in \widehat{G}$, $a\in Z(G)$ and $\d(g)=g(1-a)\a (g)$, $g \in G$,
for some $\a \in Z^1_{\chi}(G)$.

We have that $\a(gh)=\a(g)+\chi(g)\a(h)$ for all $g, h\in G$.
It follows that
$$\a(g^i)=\a(g)(1+\chi(g)+\cdots +\chi(g)^{i-1}),\ g\in G,\ i\>1,$$
and $\a(1)=0$. Hence $0=\a(1)=\a(gg^{-1})=\a(g)+\chi(g)\a(g^{-1})$,
and so $\a(g^{-1})=-\chi(g)^{-1}\a(g)=-\chi(g^{-1})\a(g)$, $g\in G$.

By the definition of Hopf-Ore extension, we have that
$$g^{-1}xg=\chi(g)x+\a (g) (1-a),\, \D(x)=x\ot a+1\ot x, \, g\in G.$$

In order to compute the rank of $H$,
we first simplify the presentation of $H=kG(\chi, a, \d)$
by changing the generator $x$, depending on the values of $\chi(a)$ and $\a(a)$.

{\bf Case 1:} $\chi(a)\neq 1$. In this case, set $x'=x-\frac{\a(a)}{1-\chi(a)}(1-a)$.
Then for any $g\in G$, we have
$$\begin{array}{rcl}
g^{-1}x'g&=&g^{-1}xg-\frac{\a(a)}{1-\chi(a)}(1-a)\\
&=&\chi(g)x+\a(g)(1-a)-\frac{\a(a)}{1-\chi(a)}(1-a)\\
&=&\chi(g)x'+(\frac{\chi(g)\a(a)}{1-\chi(a)}+\a(g)-\frac{\a(a)}{1-\chi(a)})(1-a)\\
&=&\chi(g)x'+\frac{1}{1-\chi(a)}({\chi(g)\a(a)}+\a(g)-\a(g)\chi(a)-\a(a))(1-a)\\
&=&\chi(g)x'+\frac{1}{1-\chi(a)}(\a(ga)-\a(ag))(1-a)\\
&=&\chi(g)x'.\\
\end{array}$$
$H$ is generated, as an algebra, by $G$ and $x'$. Thus, replacing the generator $x$ by $x'$,
one can assume that $g^{-1}xg=\chi(g)x$, $g\in G$,
and $\D(x)=x\ot a+1\ot x$ in this case (see \cite[Theorem 3.3(3)]{Sc}).

{\bf Case 2:} $\chi(a)= 1$ and $\a(a)=0$. In this case, $a^{-1}xa=x$.

Note that if $|a|=n<\oo$ with $\chi(a)= 1$, char$(k)=0$, or char$(k)=p>0$ and $p\nmid n$, then we have
$$0=\a(a^n)=\a(a)(1+\chi(a)+\cdots +\chi(a)^{n-1})=n\a(a),$$
which implies that $\a(a)=0$, where $|a|$ denotes the order of $a$.

{\bf Case 3:} $\chi(a)= 1$ and $\a(a)\neq 0$. In this case, $a\neq 1$.
By replacing $x$ with $\a(a)^{-1}x$, one may assume that $a^{-1}x a=x+(1-a)$ and
$\D(x)=x\ot a+1\ot x$.

Note that $H$ has a $k$-basis $\{gx^i \mid g\in G,i\>0\}$.
Let $H(n)={\rm span}\{gx^n\mid g\in G\}$ and $H[n]={\rm span}\{gx^i\mid g\in G, 0\< i \< n\}$,
$n\>0$.
Then $H[0]=H(0)=H_0=kG$ and $H[n]=\oplus_{i=0}^nH(i)$ for $n>0$. By \leref{1.6}(a),
$H[0]\subseteq H[1]\subseteq H[2]\subseteq \cdots$ is
a coalgebra filtration of $H$ (see \cite{Mo} for the definition of coalgebra
filtration of a coalgebra).

In the rest of this subsection, write $q=\chi(a)$ for simplicity.

\begin{proposition}\prlabel{2.1}
In \textbf{Cases 1} and \textbf{2}, $H=\oplus_{n=0}^{\oo}H(n)$ is a graded coalgebra.
\end{proposition}
\begin{proof}
Since $a^{-1}xa=qx$ in \textbf{Cases 1} and \textbf{2}, it follows from \cite[Proposition IV.2.2]{kas}
or \cite[Eq. (1)]{KR} that
$\D (H(n))\ss \oplus_{i+j=n}(H(i)\ot H(j))$ for all
$n\>0$. Obviously, $\varepsilon(H(n))=0$ for all $n\>1$. This completes the proof.
\end{proof}

We now analyze the term $H_1$ in the coradical filtration of $H$. Some conclusions of the following
\thref{2.2} appeared in \cite[Lemma 5]{Ra}, and some can be derived from \cite[Lemma 4.9]{Sc}.
For the completeness, we state the theorem with an independent proof as follows.

\begin{theorem}\thlabel{2.2}
{\rm (a)} Assume char$(k)=0$. If $q$ is a primitive $n$-th root of unity for some $n\>2$,
then the rank of $H$ is $2$ and $H_1=H_0\op H_0x\op H_0x^n$;
otherwise, the rank of $H$ is $1$ and $H_1=H_0\op H_0x$.\\
{\rm (b)} Assume char$(k)=p>0$.
If $q$ is a primitive $N$-th root of unity for some $N\>2$,
then the rank of $H$ is infinite and $H_1=H_0\op H_0x\op (\op_{r\>0}H_0x^{Np^r})$;
if $q=1$,
then the rank of $H$ is infinite and $H_1=H_0\op (\op_{r\>0}H_0x^{p^r})$;
otherwise, the rank of $H$ is $1$ and $H_1=H_0\op H_0x$.
\end{theorem}
\begin{proof}
By \cite[Theorem 5.4.1]{Mo} we know that $H_1$ is spanned by all group-like elements
and skew primitive elements. Hence we need to compute all $(g,h)$-primitive
elements for all $g, h\in G$.
Note that $y\in P_{g,h}(H)$ if and only if $h^{-1}y\in P_{h^{-1}g,1}(H)$.
Hence we only need to describe
the $(g,1)$-primitive elements for any $g\in G$.

We first consider \textbf{Cases 1} and \textbf{2}. In these cases,
 we may assume that
$$a^{-1}xa=qx,\quad \D(x)=x\ot a+1\ot x.$$

Let $0\neq h \in P_{g, 1}(H)$ for some $g\in G$.
By \prref{2.1}, $H$ is a graded coalgebra, and hence $P_{g, 1}(H)=\bigoplus_{i=0}^{\infty}(P_{g, 1}(H)\cap H(i))$.
Thus, we only need to consider the case that
$h\in H(i)=H_0x^i$, $i\>0$.  If $h\in H_0$, then $h=\beta (1-g)$ for some $\beta \in k^{\times}$.
If $h\in H_0x$, then $h=\sum _{b\in G}\beta _b bx$, $\beta _b\in k$.
Comparing the both sides of the equation
$$\sum _{b\in G}\b_b(bx\ot ba +b\ot bx)=
(\sum _{b\in G}\b_b bx)\ot g +1\ot (\sum_{b\in G}\b_b bx),$$
we have that $\b_b=0$ if $b\neq 1$. Hence $h=\b_1x$, and so $g=a$. Thus,
$$h=\b_1x,\quad \D(h)=h\ot a+1\ot h.$$
If $h\in H_0x^n$ for some $n\>2$, then $h=\sum_{b\in G}\b_b bx^n$, $\b_b\in k$. Hence we have
$$\sum _{b\in G}\b_b(\sum_{l=0}^n\binom{n}{l}_qbx^{n-l}\ot ba^{n-l}
 x^l)=(\sum _{b\in G}\b_b bx^n)\ot g +1\ot (\sum_{b\in G}\b_b bx^n),$$
from which one knows that $\b_b=0$ if $b\neq 1$. Thus, $h=\b_1x^n$,
$\binom{n}{l}_q=0$ for all
$1\<l\<n-1$, and $g=a^n$. By \cite[Corollary 2]{Ra}, char$(k)=0$ and $q$ is a primitive $n$-th root of unity,
or char$(k)=p>0$ and $q$ is a primitive $N$-th root of unity, where $n=Np^r$ with
$p\nmid N$, $N\>1$ and $r\>0$.
This shows the theorem for \textbf{Cases 1} and \textbf{2}.

Now we consider \textbf{Case 3}. In this case,
$a^{-1}xa=x+(1-a)$ and $\D(x)=x\ot a+1\ot x$, where $a\neq 1$.
Let $0\neq h=\sum_{i=0}^n\sum_{b\in G}\b(i,b)bx^i \in H$, where $\b(i,b)\in k$,
and $\b(n, b)\neq 0$ for some $b\in G$.
Assume $\D(h)=h\ot g+1\ot h$ for some $g\in G$. Then
$$\begin{array}{rl}
&\sum\limits_{i=0}^n\sum\limits_{b\in G}\sum\limits_{l=0}^i\b(i,b)\binom{i}{l}
bx^{i-l}\ot ba^{i-l}x^l+\sum\limits_{i=0}^n\sum\limits_{b\in G}\b(i,b)(b\ot b)P_i\\
=&\sum\limits_{i=0}^n\sum\limits_{b\in G}\b(i,b)bx^i\ot g+1\ot \sum\limits_{i=0}^n\sum\limits_{b\in G}\b(i,b)bx^i,\\
\end{array}$$
where $P_i$'s are the elements described as in \leref{1.6}(a).
Note that $\{x^i\ot x^j\mid i,j\>0\}$ is a
basis of $H\ot H$ as a free $H_0\ot H_0$-module, and
$\{bx^i\ot cx^j\mid i,j\>0, \, b,c\in G\}$ is a $k$-basis of $H\ot H$.
So we can compare the coefficients of the two sides of the above equation.

(1) If $n=0$, then obviously $h=\b(1-g)$ for some $\b\in k^{\times}$.

Now we consider the case of $n\>1$. By comparing the coefficients of $1\ot x^n$,
one gets that $\b(n, b)=0$ for all $b\neq 1$. Hence $\b(n,1)\neq 0$.
By comparing the coefficients of $x^n\ot 1$, one gets that $g=a^n$.

(2) If $n=1$, then $g=a$ and $h=h_0+\b x\in H[1]$ for some $h_0\in H(0)$ and $\b\in k^{\times}$.
Since $H[1]=H(0)\op H(1)$ is a graded coalgebra,
a similar argument as above shows that $h=\b_0(1-a)+\b x$ for some $\b_0\in k$.

$(3)$ If $n\>2$, by comparing the coefficients of $x^{n-l}\ot x^l$ ($0<l<n$),
one gets that $\b(n,1)\binom{n}{l}=0$ for all
$0<l<n$. This forces char$(k)=p>0$ and $n=p^r$
for some $r\>1$ by \cite[Corollary 2]{Ra}. Set $n_1=n$, $r_1=r$, $h_1=h$. Then
$h_2:=h-\b(p^{r_1},1)(x^{p^{r_1}}-x^{p^{r_1-1}})$ is an
$(a^{p^{r_1}},1)$-primitive element by \leref{1.6}(b). If $0\neq
h_2=\sum_{i=0}^{n_2}\sum_{b\in G}\g(i,b)bx^i$ such that
$\g(n_2,b)\neq 0$ for some $b\in G$, and $n_2\>2$, then a similar
discussion as above shows that $n_2=p^{r_2}$ for some
$r_2\>1$ and $a^{p^{r_1}}=a^{p^{r_2}}$. Continuing this procedure, we
have that $h-\sum_{i=1}^s\b_i(x^{p^{r_i}}-x^{p^{r_i-1}})\in H[1]$,
where $\b_i\in k^{\ti}$, $r_1>r_2>\cdots >r_s\>1$,
$s\>1$, and $a^{p^{r_1}}=a^{p^{r_2}}=\cdots =a^{p^{r_s}}$.
If $h-\sum_{i=1}^s\b_i(x^{p^{r_i}}-x^{p^{r_i-1}})=0$,
then $h=\sum_{i=1}^s\b_i(x^{p^{r_i}}-x^{p^{r_i-1}})$.
If $h-\sum_{i=1}^s\b_i(x^{p^{r_i}}-x^{p^{r_i-1}})\neq 0$, then
it follows from (1) and (2) that either $h=\sum_{i=1}^s\b_i(x^{p^{r_i}}-x^{p^{r_i-1}})+\b_0(1-g)$
with $\b_0\in k^{\times}$ and $g=a^{p^{r_1}}=a^{p^{r_2}}=\cdots =a^{p^{r_s}}$,
or $h=\sum_{i=1}^s\b_i(x^{p^{r_i}}-x^{p^{r_i-1}})+\b x+\b_0(1-a)$
with $\b\in k^{\times}$, $\b_0\in k$ and $a=a^{p^{r_1}}=a^{p^{r_2}}=\cdots =a^{p^{r_s}}$.
Note that the second possibility does not occur since $a\neq a^{p^r}$ for any $r\>1$
in this case.
This shows the theorem for \textbf{Case 3} by \leref{1.6}(b).
\end{proof}

By \thref{2.2} and its proof, we have the following corollary (see
\cite[Lemma 5]{Ra}).

\begin{corollary}\colabel{2.3}
Let $g\in G$ and $z\in H\backslash H_0$ with $\D (z)=z\ot g+1\ot z$.\\
{\rm(a)} In the case of char$(k)=0$, we have\\
\mbox{\hspace{0.4cm} \rm(i) } if $q$ is a primitive $n$-th root of unity with $n\>2$,
then either $g=a,\, z=\gamma x+\b (1-a)$,\\
\mbox{\hspace{0.8cm} } or $g=a^n$, $z=\gamma x^n+\b (1-a^n)$, where $\gamma, \b \in k$;\\
\mbox{\hspace{0.3cm} \rm(ii) } Otherwise, $g=a$, $z=\gamma x+\b (1-a)$ for some
$\gamma, \b \in k$.\\
{\rm(b)} In the case of char$(k)=p>0$, we have\\
\mbox{\hspace{0.4cm} \rm(i) } if $q$ is a primitive $N$-th root of unity with $N\>2$, then
either $g=a$, $z=\gamma x+\b(1-a)$,\\
\mbox{\hspace{0.8cm} } or $g=a^{Np^r}$, $z=\gamma(1-a^{Np^r})+\sum_{i=1}^s\b_ix^{Np^{r_i}}$,
where $\gamma, \b, \b_1, \cdots, \b_s\in k$, $r, r_1, \cdots, r_s\>0$\\
\mbox{\hspace{0.8cm} } such that $a^{Np^{r_i}}=a^{Np^r}$;\\
\mbox{\hspace{0.3cm} \rm(ii) } if $q=1$ and $a^{-1}xa=x$, then $g=a^{p^r}$,
$z=\gamma (1-a^{p^r})+\sum_{i=1}^s\b_ix^{p^{r_i}}$
for some $\gamma,\b_1,\cdots,$\\
\mbox{\hspace{0.8cm} } $\b_s\in k$ and
$r, r_1, \cdots, r_s\>0$ such that $a^{p^{r_i}}=a^{p^r}$;\\
\mbox{\hspace{0.2cm} \rm(iii) } if $q=1$ and $a^{-1}xa=x+(1-a)$,
then $g=a^{p^r}\neq a$, $z=\gamma (1-a^{p^r})+\sum_{i=1}^s\b_i(x^{p^{r_i}}-x^{p^{r_i-1}})$\\
\mbox{\hspace{0.8cm} } for some $\gamma, \b_1, \cdots, \b_s\in k$ and
$r, r_1, \cdots, r_s\>0$ such that $r_i\>1$ and $a^{p^{r_i}}=a^{p^r}$;\\
\mbox{\hspace{0.2cm} \rm (iv) } if $q$ is not a root of unity, then $g=a$, $z=\gamma x+\b (1-a)$ for some
$\gamma, \b \in k$.
\end{corollary}

\subsection{The classification of pointed Hopf algebras of rank one}
\cite{KR} and \cite{Sc} classified all finite dimensional pointed Hopf algebras of rank one over
an algebraically closed field with char$(k)=0$ and char$(k)=p>0$, respectively.
Actually, the most classification results there still hold without the assumptions
that $k$ is algebraically closed and the Hopf algebras are
finite dimensional. In this subsection, we describe the structures of
(infinite or finite dimensional)
pointed Hopf algebras of rank one over an arbitrary field $k$.

\begin{lemma}\lelabel{2.4}
Let $H$ be a pointed Hopf algebra of rank one, $G=G(H)$.
Then there are some $a\in G$ and
$x\in H\setminus H_0$ such that $\D(x)=x\ot a+1\ot x$. Moreover, $H_1=H_0\op H_0x$.
\end{lemma}
\begin{proof}
Since $H$ is a pointed Hopf algebra of rank one, $H_0\neq H_1$.
It follows from \cite[Proposition 2]{TaWi} that
there are some $a\in G$ and
$x\in H_1\setminus H_0$ such that $\D(x)=x\ot a+1\ot x$
(see \cite[Proposition 1(a)]{KR} and its proof). Then from the paragraph
after \cite[Remark 2.7]{Sc}, one knows that $\{1, x\}$ forms a basis of $H_1$
as a left $H_0$-module, and so $H_1=H_0\op H_0x$.
\end{proof}

In the rest of this section, assume that $H$ is an arbitrary (infinite or finite dimensional)
pointed Hopf algebra of rank one, and $G=G(H)$. In this case, $H_0=kG$.
By \leref{2.4}, there exist some $a\in G$ and $x\in H_1\setminus H_0$
such that $H_1=H_0 \op H_0x$ and $\D(x)=x\ot a+1\ot x$. Moreover, $H$ is generated by $x$ and $G$
as an algebra, $H_1=H_0 \op H_0x$ is a graded coalgebra, and $\{gx^i\mid g\in G,0\<i\<1\}$
is a $k$-basis of $H_1$.

The following \prref{2.5} is contained in \cite[Proposition 1]{KR} and \cite[Lemma 2.8]{Sc}
when $H$ is finite dimensional.

\begin{proposition}\prlabel{2.5}
{\rm (a)} If $g\in G$ and $z\in H\setminus H_0$ with $\D(z)=z\ot g +1\ot z$, then $g=a$ and
$z=\gamma x+\b (1-a)$ for some $\gamma\in k^{\times}$ and $\b\in k$.\\
{\rm (b)} $a\in Z(G)$.\\
{\rm (c)} There exist a character $\chi \in \widehat{G}$ and a $1$-cocycle $\a \in Z_{\chi}^1(G)$
such that $$g^{-1}xg=\chi (g)x+\a(g)(1-a),\quad g\in G.$$
In particular, if $a=1$ then $g^{-1}xg=\chi (g)x$, $g\in G.$
\end{proposition}
\begin{proof}
{ \rm (a)} Let $g\in G$ and $z\in H\setminus H_0$ with $\D(z)=z\ot g+1\ot z$. Then $z\in H_1$.
Since $H_1=H_0\op H_0x$ is a graded coalgebra, $P_{g, 1}(H_1)=(P_{g, 1}(H_1)\cap H_0)\oplus(P_{g, 1}(H_1)\cap H_0x)$.
Hence we only need to consider the case of $z\in H_0x$.
Let $z\in H_0x$. Then $z=\sum_{b\in G}\gamma_b bx$ for some $\gamma_b\in k$, and hence
$$\sum_{b\in G}\gamma_b(bx\ot ba+b\ot bx)=(\sum_{b\in G}\gamma_b bx)\ot g+1\ot (\sum_{b\in G}\gamma _bbx).$$
It follows that $g=a$ and $\gamma _b=0$ for all $b\neq 1$. Thus $z=\gamma_1 x$ with $0\neq\gamma_1\in k$.

{\rm (b),(c)} For any $g\in G$, we have that $\D(g^{-1}xg)=g^{-1}xg\ot g^{-1}ag+1\ot g^{-1}xg$.
By {\rm (a)}, one knows that $g^{-1}ag=a$ and $g^{-1}xg=\chi(g)x+\a (g)(1-a)$
for some $\chi(g),\ \a(g)\in k$ with $\chi(g)\neq 0$. Hence $a\in Z(G)$.
Obviously, if $a=1$ then $\chi\in \widehat{G}$ and one can take $\a(g)=0$, $g\in G$.
Now assume that $a\neq 1$. We claim that $\chi\in \widehat{G}$ and $\a\in Z_{\chi}^1(G)$.  In fact,
let $g,h\in G$. Then, on one hand, $(gh)^{-1}x(gh)=\chi (gh)x+\a(gh)(1-a)$. On the other hand,
\begin{align*}
(gh)^{-1}x(gh)=&h^{-1}(g^{-1}xg)h\\
 =&h^{-1}(\chi(g)x+\a (g)(1-a))h\\
=&\chi(g)(\chi(h)x+\a (h) (1-a))+\a(g)(1-a)\\
=&\chi(g)\chi(h)x+(\chi(g)\a (h)+\a (g))(1-a).
\end{align*}
It follows that $\chi(gh)=\chi(g)\chi(h)$, and hence $\chi \in \widehat{G}$.
Since $a\neq 1$, we also have $\a(gh)=\chi(g)\a(h)+\a (g)$, which means that
$\a\in Z_{\chi}^1(G)$.
\end{proof}

From \prref{2.5}(c), we have
$$g^{-1}xg=\chi(g)x+\a(g)(1-a), \ g\in G$$
in $H$. By a discussion similar to the one about the Hopf-Ore extension before,
one can classify $H$ in three types according to the values of $\chi(a)$ and $\a(a)$
(see \cite[Definition 4.2]{Sc}).\\
\mbox{\hspace{0.5cm}{\bf Type 1:}} $\chi(a)\neq 1$. In this case, $a\neq 1$,
by setting $\gamma=\frac{\a(a)}{1-\chi(a)}$ and substituting $x$ with
\mbox{\hspace{1.6cm}} $x-\gamma(1-a)$,
one may assume that
 $$g^{-1}xg=\chi(g)x, \ g\in G.$$
\mbox{\hspace{0.5cm}{\bf Type 2:}} $\chi(a)=1$ and $\a(a)=0$. In this case, we have
$$a^{-1}xa=x.$$
\mbox{\hspace{0.5cm}{\bf Type 3:}} $\chi(a)=1$ and $\a(a)\neq 0$. In this case, $a\neq 1$ and
$a^{-1}xa=x+\a(a)(1-a)$. By
\mbox{\hspace{1.6cm}} replacing $x$ with $\a(a)^{-1}x$, one may assume that
\begin{equation*}
   a^{-1}xa=x+(1-a),\quad \D (x)=x\ot a+1\ot x.
\end{equation*}

When $G=\langle a\rangle$ is a cyclic group or $G$ is a finite group, the following
\prref{2.6} appears in \cite[Lemma 5]{Ra} and \cite[Theorem 3.11 and Lemma 4.12]{Sc}.
Actually, the proofs there are still valid for the general case.

\begin{proposition}\prlabel{2.6}
$H$ is a free left $H_0$-module under left multiplication with a basis $\{1, x, x^2, \cdots \}$ or
$\{1, x, x^2, \cdots, x^{n-1}\}$ for some $n\>2$, where $x$ is the $(a, 1)$-primitive
element in $H$ as given above. Furthermore, in case that $\{1,x,x^2,\cdots ,x^{n-1}\}$ is
an $H_0$-basis of the free left $H_0$-module $H$, where $n\>2$, we have the following statements:\\
{\rm (a)} If char$(k)=0$, then $\chi(a)$ is a primitive $n$-th root of unity, and by a suitable substituting of $x$,
we have that $g^{-1}xg=\chi(g)x$ for all $g\in G$, and $x^n=\b (1-a^n)$ for some $\b \in k$.\\
{\rm (b)} If char$(k)=p>0$, then either $p\nmid n$ and $\chi(a)$ is a primitive $n$-th root of unity,
or $n=p$ and $\chi(a)=1$. Moreover, we have\\
\mbox{\hspace{0.4cm} (i)} if $p\nmid n$, then by a suitable substituting of $x$,
we have that $g^{-1}xg=\chi(g)x$ for all $g\in$
\mbox{\hspace{0.8cm}} $G$, and $x^n=\b (1-a^n)$ for some $\b \in k$;\\
\mbox{\hspace{0.3cm} (ii)} if $n=p$ and $a^{-1}xa=x$, then $x^p=\b x+\gamma (1-a^p)$
for some $\b, \gamma \in k$. Moreover, $a^p=a$
\mbox{\hspace{0.8cm}} if $\b\neq 0$;\\
\mbox{\hspace{0.2cm} (iii)} if $n=p$ and $a^{-1}xa=x+\a(a)(1-a)$ with $\a(a)\neq 0$,
then by a suitable substituting
\mbox{\hspace{0.9cm}} of $x$, we have $x^p=x+\g(1-a^p)$ for some $\g \in k$.
\end{proposition}
\begin{proof}
An argument similar to the proof of \cite[Lemma 1(a)]{KR} shows
that $H$ is a free left $H_0$-module under the left multiplication with a basis $\{1, x, x^2, \cdots \}$ or
$\{1, x, x^2, \cdots, x^{n-1}\}$ for some $n\>2$.
Then Parts (a) and (b)(i, ii) follow from the proof of \cite[Lemma 5]{Ra} or
\cite[Theorem 3.11]{Sc}, and Part (b)(iii) follows from \cite[Lemma 4.12]{Sc}.
\end{proof}

We have already described the algebra structure of $H$ in Propositions 2.5 and 2.6.
Now we give another description of $H$ using the Hopf-Ore extension as follows.
For any $y\in H$, let $\langle y\rangle$ denote the ideal of $H$ generated by $y$.

\begin{theorem}\thlabel{2.7}
$H$ is isomorphic to a quotient of Hopf-Ore extension of
its coradical $H_0=kG$, where $G=G(H)$.
Precisely, there exist a $\chi \in \widehat{G}$, $a\in Z(G)$ and $\a \in Z_{\chi}^1(G)$
such that $H\cong kG(\chi ,a, \delta)/I$ as Hopf algebras, where the
derivation $\delta$ is defined by $\delta(g)=g(1-a)\a(g)$, $g\in G$, $I$ is a Hopf ideal
of $kG(\chi ,a, \delta)$ defined as follows:\\
{\rm(a)} In the case of char$(k)=0$,\\
\mbox{\hspace{0.4cm} (i)} if $\chi(a)$ is a primitive $n$-th root of unity for some $n\>2$, then
$I=\langle x^n-\b(1-a^n)\rangle$ for
\mbox{\hspace{0.8cm}} some $\b\in k$;\\
\mbox{\hspace{0.3cm} (ii)} if $\chi(a)$ is not a primitive $n$-th root of unity for any $n\>2$,
then $I=0$.\\
{\rm(b)} In the case of char$(k)=p>0$,\\
\mbox{\hspace{0.4cm} (i)} if $\chi(a)$ is a primitive $n$-th root of unity for some $n\>2$, then
$I=\langle x^n-\b(1-a^n)\rangle$ for
\mbox{\hspace{0.8cm}} some $\b\in k$;\\
\mbox{\hspace{0.3cm} (ii)} if $\chi(a)=1$ and $a^{-1}xa=x$, then $I=\langle x^p-\b x-\gamma (1-a^p)\rangle$
for some $\b$, $\gamma \in k$. Moreover,
\mbox{\hspace{0.8cm}} $a^p=a$ if $\b\neq 0$;\\
\mbox{\hspace{0.2cm} (iii)} if $\chi(a)=1$ and $a^{-1}xa=x+\a(a)(1-a)$ with $\a(a)\neq 0$,
then by a suitable  substituting
\mbox{\hspace{0.8cm}} of $x$, $I=\langle x^p-x-\g(1-a^p)\rangle$ for some $\g\in k$;\\
\mbox{\hspace{0.2cm} (iv)} if $\chi(a)$ is not a primitive $n$-th root of unity for any $n\>1$,
$I=0$.
\end{theorem}
\begin{proof}
Let $G=G(H)$. Then $H_0=kG$. By \leref{2.4}, there exist some $a\in G$ and $x\in H_1\setminus H_0$
such that $H_1=H_0 \op H_0x$ and $\D(x)=x\ot a+1\ot x$. Moreover, $H$ is generated by $x$ and $G$
as an algebra. By \prref{2.5}, $a\in Z(G)$, and there exist a character $\chi \in \widehat{G}$
and a $1$-cocycle $\a \in Z_{\chi}^1(G)$
such that $$g^{-1}xg=\chi (g)x+\a(g)(1-a),\quad g\in G.$$
Moreover, we can make the following assumption: $\a(g)=0$ for all $g\in G$
when $\chi(a)\neq 1$; $\a(a)=1$ when $\chi(a)=1$ and $\a(a)\neq 0$.
Then one can form a Hopf-Ore extension $kG(\chi ,a, \delta)$ of $kG$,
where $\delta\in{\rm End}_k(kG)$ is defined by $\delta(g)=g(1-a)\a(g)$ for all $g\in G$.
By the definition of Hopf-Ore extensions, there exists an algebra epimorphism
$F: kG(\chi ,a, \delta)\ra H$ such that $F(x)=x$ and $F(g)=g$ for all $g\in G$. It is
easy to see that $F$ is also a coalgebra morphism. Consequently, $F$ is a Hopf algebra epimorphism.
Note that $kG(\chi ,a, \delta)$ is a free left $kG$-module with the $kG$-basis $\{x^i|i\>0\}$.
Obviously, $F$ is a $kG$-module homomorphism.
Let $I={\rm Ker}(F)$. Then $I$ is a Hopf ideal of $kG(\chi ,a, \delta)$, and
$H\cong kG(\chi ,a, \delta)/I$ as Hopf algebras.

{ \rm (a)(i)} Assume that $\chi(a)$ is a primitive $n$-th root of unity for some $n\>2$.
Then $\D(x^n)=x^n\ot a^n+1\ot x^n$ in $H$, and hence $x^n\in H_1=H_0\op H_0x$.
It follows from  \prref{2.6} that $H$ is a free left
$H_0$-module with a basis $\{1, x, x^2, \cdots, x^{m-1}\}$
for some $2\<m\<n$. By \prref{2.6}(a),
$\chi(a)$ is a primitive $m$-th root of unity. Hence $m=n$ and
$H$ has a left $H_0$-basis $\{1, x, x^2, \cdots, x^{n-1}\}$. Moreover,
$x^n=\b(1-a^n)$ in $H$ for some $\b\in k$. Hence $I':=\langle x^n-\b(1-a^n)\rangle\subseteq I$, and so
$F$ induces an algebra epimorphism $\widetilde{F}: kG(\chi ,a, \delta)/I'\ra H$,
which is also a $kG$-module homomorphism. Obviously, $kG(\chi ,a, \delta)/I'$
is generated as a $kG$-module by $\{x^i+I'|0\<i\<n-1\}$.
Since $\{x^i|0\<i\<n-1\}$ is a $kG$-basis of $H$ and $\widetilde{F}(x^i+I')=x^i$,
$\{x^i+I'|0\<i\<n-1\}$ is a $kG$-basis of $kG(\chi ,a, \delta)/I'$.
It follows that $\widetilde{F}$ is an isomorphism and so $I=\langle x^n-\b(1-a^n)\rangle$.

{ \rm (a)(ii)}  If $\chi(a)$ is not a primitive $n$-th root of unity for any $n\>2$,
then $H$ is a free left
$H_0$-module with the basis $\{1, x, x^2, \cdots \}$ by \prref{2.6}.
Obviously, $F$ is an isomorphism in this case. Hence $I=0$ and $H\cong kG(\chi ,a, \delta)$
as Hopf algebras.

{ \rm (b)} It is similar to Part (a) using \prref{2.6}. In this case, the formulas for $I$ also follow
from \cite[Theorem 3.11]{Sc}.
\end{proof}

Note that the construction of the datum $\{\chi, a, \d\}$ associated to $H$ of {\bf type 1}
is given in \cite[Lemma 2.8]{Sc}. In the following,
let $|\chi|$ denote the order of the character $\chi$.

\begin{proposition}\prlabel{2.8}
With the notations in \thref{2.7},
assume that $\chi(a)$ is a primitive $n$-th root of unity for some $n\>2$. Then we have\\
{ \rm (a)} $n\<|\chi|$.\\
{ \rm (b)} If $|\chi|<\oo$, then $n||\chi|$.\\
{ \rm (c)} If $n<|\chi|$, then $\b=0$ or $a^n=1$, and consequently $I=\langle x^n\rangle$.
\end{proposition}
\begin{proof}
Parts (a) and (b) are obvious. For Part (c), we may assume that $g^{-1}xg=\chi(g)x$
for all $g\in G$ as stated in the proof of \thref{2.7}. Since $n<|\chi|$, there exists some $g\in G$
such that $\chi^n(g)\neq 1$. Then $x^n\in I$ since $x^n-\b(1-a^n)\in I$ and
$\chi^n(g)x^n-\b(1-a^n)=g^{-1}(x^n-\b(1-a^n))g\in I$.
Hence $I=\langle x^n-\b(1-a^n)\rangle=\langle x^n, \b(1-a^n)\rangle$.
If $\b\neq 0$, then $1-a^n\in I$, and hence $a^n=1$ in $H=kG(\chi, a, \d)/I$.
Since $G=G(H)$, $a^n=1$ in $G\subset kG(\chi, a, \d)$.
This implies $\b(1-a^n)=0$, and so $I=\langle x^n\rangle$.
\end{proof}

\begin{corollary}\colabel{2.9}
With the notations in \thref{2.7},
assume that $\chi(a)$ is a primitive $n$-th root of unity for some $n\>2$. Then we have\\
{  \rm (a)} If $\b=0$ or $a^n=1$, then $I=\langle x^n \rangle$.\\
{  \rm (b)} If $\b\neq 0$ and $a^n\neq 1$, then $|\chi|=n$ and $I=\langle x^n-\b(1-a^n)\rangle$.
\end{corollary}
\begin{proof}
It follows from \thref{2.7} and \prref{2.8}.
\end{proof}
\begin{remark}
Let $G=G(H)$ and $H_0=kG$. When $\chi(a)$ is a primitive $n$-th root of unity for some $n\>2$,
$H=H_0+H_0x+H_0x^2+\cdots +H_0x^{n-1}$.
When char$k=p>0$ and $\chi(a)=1$,
$H=H_0+H_0x+H_0x^2+\cdots +H_0x^{p-1}$. In these cases,
if $H$ is infinite dimensional, then $G$ is an infinite group,
and $H$ has the similar structures and properties
with the finite dimensional pointed Hopf algebras of rank one classified
in \cite{KR,Sc} except that $G$ is infinite.
\end{remark}

\section{Simple weight modules}\selabel{3}

In this and next sections, we study the representation theory of Hopf-Ore extensions of group algebras and pointed
Hopf algebras of rank one. By the discussion
in \seref{2}, any pointed Hopf algebra of rank one is isomorphic to
a quotient of the Hopf-Ore extension of
its coradical.

In what follows, let $H=kG(\chi^{-1},a,\d)$ be a Hopf-Ore extension of
a group algebra $kG$. Then $H'=kG(\chi^{-1},a,\d)/I$ is a Hopf algebra of rank one,
where $I$ is a Hopf ideal of $H$ as described in \thref{2.7}. Moreover, every
pointed Hopf algebra of rank one has this form. Let $\pi: H\ra H'$
be the canonical epimorphism.

Let ${\mathcal M}$ (resp. ${\mathcal M}'$) denote the category
of the left $H$-modules (resp. $H'$-modules). Then both ${\mathcal M}$ and ${\mathcal M}'$
are monoidal categories. Since $H'=H/I$ is a quotient Hopf algebra of $H$,
there is a monoidal category embedding functor from ${\mathcal M}'$ to ${\mathcal M}$.
Thus, one can regards ${\mathcal M}'$ as a monoidal full subcategory of ${\mathcal M}$.
Moreover, an object $M$ in ${\mathcal M}$ is an object in the subcategory ${\mathcal M}'$
if and only if $I\cdot M=0$. We present the observation as a proposition below.

\begin{proposition}\prlabel{3.1}
${\mathcal M}'$ is a monoidal full subcategory of ${\mathcal M}$ consisting of all
those $H$-modules $M$ such that $I\cdot M=0$.
\end{proposition}

From now on, we assume that $\chi^{-1}(a)\neq 1$, i.e. $H$ is of \textbf{Case 1} (or $H'$ is of \textbf{Type 1}).
In this case, $\chi$ is not the trivial character.
We may assume that $\d=0$ and $H=kG(\chi^{-1},a,0)$. That is, $H$ is generated, as an algebra,
by $G$ and $x$ subject to the
relation: $xg=\chi^{-1}(g)gx$, $g\in G$. The coalgebra structure and antipode are given
by
$$\begin{array}{lll}
\D(x)=x\ot a+1\ot x,& \ep(x)=0,& S(x)=-xa^{-1},\\
\D(g)=g\ot g,& \ep(g)=1,& S(g)=g^{-1},\\
\end{array}$$
where $g\in G$.
$H$ has a $k$-basis $\{gx^i \mid g\in G,i\>0\}$.
Let $H(n)={\rm span}\{gx^n\mid g\in G\}$ for all $n\>0$. Then $H=\oplus_{n=0}^{\oo}H(n)$ is a
graded coalgebra, and $H(0)=H_0=kG$ is the coradical of $H$ by \prref{2.1}.
Obviously, $H$ is also a graded algebra and $S(H(n))\subseteq H(n)$ for all $n\>0$.
Thus, $H$ is a graded Hopf algebra (see \cite{Mo} for the definition of graded Hopf algebra).

Let $V$ be an $H$-module. Then $V$ becomes an $H_0$-module by the restriction. Conversely,
for an $H_0$-module $W$, $W$ becomes an $H$-module by setting $x\cdot w=0$, $w\in W$.
Moreover, $W$ is a graded $H$-module with the trivial grading given by
$W_{(0)}=W$ and $W_{(i)}=0$ for $i>0$.

Let $V=\op_{i=0}^{\oo}V_{(i)}$ be a graded $H$-module.
Then $g\cdot V_{(i)}=V_{(i)}$ and $x\cdot V_{(i)}\subseteq V_{(i+1)}$
for all $g\in G$ and $i\>0$.
Let $V_{[n]}=V_{(n)}\op V_{(n+1)}\op V_{(n+2)}\op \cdots=\op _{i=n}^{\oo}V_{(i)}$. Then
$V=V_{[0]}\supseteq V_{[1]}\supseteq V_{[2]}\supseteq \cdots$
is a descending chain of graded $H$-submodules.
If $V$ is simple, then there exists some $n\>0$ such that
$V=V_{[0]}=V_{[1]}=V_{[2]}=\cdots =V_{[n]}=V_{(n)}$,
and $V_{[n+1]}=0$. Hence $x\cdot V=0$ in this case. Thus, we have
the following proposition.

\begin{proposition}\prlabel{3.2}
Let $V$ be a graded $H$-module. Then $V$ is simple as an $H$-module
if and only if $V$ is simple as an $H_0$-module.
\end{proposition}
\begin{proof}
Assume that $V$ is simple as an $H$-module. Then $x\cdot V=0$
from the discussion above. Hence each $H_0$-submodule
of $V$ is an $H$-submodule. It follows that
$V$ is also simple as an $H_0$-module. The converse is obvious.
\end{proof}

Let $M\in$ ${\mathcal M}$. For any $\l \in \widehat{G}$,
let $M_{(\l)}=\{v\in M\mid g\cdot v=\l (g)v,\, g\in G\}$.
Each nonzero element in $M_{(\l)}$ is called a {\it weight vector of weight $\l$} in $M$.
One can check that $\op _{\l \in \widehat{G}} M_{(\l)}$
is a submodule of $M$.
Let $\Pi(M)=\{\l \in \widehat{G}\mid M_{(\l)}\neq 0\}$, which is
called the {\it weight space} of $M$.
$M$ is said to be a {\it weight module} if $M=\op _{\l\in \Pi(M)}
M_{(\l)}$. Note that submodules and quotient modules of a weight module are both weight modules.

Let $\mathcal W$ be the full subcategory of ${\mathcal M}$ consisting of all weight modules.
Similarly, one can define weight modules over $H'$. Let $\mathcal W'$ be the full subcategory
of ${\mathcal M}'$ consisting of all weight modules.
Then one can easily see that $\mathcal W$ (resp. $\mathcal W'$) is a monoidal full subcategory of
${\mathcal M}$ (resp. ${\mathcal M}'$). Moreover, $\mathcal W'={\mathcal M}'\cap\mathcal W$.

For any $\l \in \widehat{G}$, let $V_{\l}$ be the $1$-dimensional $H$-module defined by
$$g\cdot v=\l(g)v \mbox{ and } x\cdot v=0,\, v\in V_{\l}.$$
The following lemma is obvious.

\begin{lemma}\lelabel{3.3}
Let $\sigma ,\l \in \widehat{G}$. Then\\
{\rm (a)} $V_{\l}$ is a simple weight $H$-module.\\
{\rm (b)} $V_{\sigma}\cong V_{\l}$ if and only if $\sigma=\l$.\\
\end{lemma}

Let $\l\in \widehat{G}$. Then one can define a Verma module $M(\l)=H\ot _{H_0}V_{\l}$.
Note that $H$ is a free right $H_0$-module and
$H=\oplus_{i=0}^{\oo}(x^iH_0)$ with $x^iH_0\cong H_0$ as right $H_0$-modules. Hence
$$M(\l)
=\oplus_{i=0}^{\oo}(x^iH_0)\ot_{H_0}V_{\l}=\oplus_{i=0}^{\oo}x^i\ot_{H_0}V_{\l}$$
as $k$-spaces.
Fix a nonzero element $v\in V_{\l}$, and let $v_{\l}:=1\ot_{H_0}v$. Then
$x^i\ot_{H_0}v=x^i\cdot(1\ot _{H_0}v)=x^i\cdot v_{\l}$.
Hence $M(\l)$ has a $k$-basis $\{x^i\cdot v_{\l}\mid i\>0\}$,
and so $M(\l)$ is a free $k[x]$-module of rank one. Moreover, $M(\l)=k[x]\cdot v_{\l}$,
where $k[x]$ is the subalgebra of $H$ generated by $x$, which is a polynomial algebra
in one variable $x$ over $k$.

\begin{proposition}\prlabel{3.4}
Let $\l, \tau \in \widehat{G}$. Then\\
{\rm (a)} $M(\l)$ is an indecomposable weight module.\\
{\rm (b)} $x\cdot M(\l)$ is a maximal submodule of $M(\l)$ and $M(\l)/(x\cdot M(\l))\cong V_{\l}$.\\
{\rm (c)} $M(\l)\cong M(\tau)$ if and only if $\l=\tau$.\\
\end{proposition}
\begin{proof}
(a) Since $\{x^i\cdot v_{\l}\mid i\>0\}$ is a
basis of $M(\l)$ and each $x^i\cdot v_{\l}$ is a weight vector,
$M(\l)$ is a weight module. Now suppose
that $P$ is a module endomorphism of $M(\l)$ with $P^2=P$.
If $P(v_{\l})=\sum_{i=0}^m\b_ix^i\cdot v_{\l}$ for some $\b_i\in k$, then from
$P^2(v_{\l})=P(v_{\l})$, one gets that $P(v_{\l})=\b_0 v_{\l}$ with $\b_0^2=\b_0$.
Thus, Im$(P)=M(\l)$ or Im$(P)=0$. Hence $M(\l)$ is indecomposable.

(b) It is obvious.

(c) If $\l=\tau$, then obviously $M(\l)\cong M(\tau)$. Conversely,
assume that $f: M(\l)\ra M(\tau)$ is an $H$-module isomorphism. Then
$f(x\cdot M(\l))=x\cdot f(M(\l))=x\cdot M(\tau)$. Hence $f$ induces
an $H$-module isomorphism $\overline{f}: M(\l)/(x\cdot M(\l))\ra M(\tau)/(x\cdot M(\tau))$.
Thus, it follows from Part (b) and \leref{3.3}(b) that $\l=\tau$.
\end{proof}

\begin{lemma}\lelabel{3.5}
Let $M\in$ ${\mathcal M}$ and $\l\in \widehat{G}$. If $v$ is a weight vector
of weight $\l$ in $M$, then there exists an $H$-module
map $\phi$ from $M(\l)$ to $M$ such that $\phi(v_{\l})=v$. Furthermore, if $M=H\cdot v$
then $\phi$ is an epimorphism.
\end{lemma}
\begin{proof}
Assume that $v$ is a weight vector of weight $\l$ in $M$. Since $\{x^i\cdot v_{\l}|i\>0\}$
is a $k$-basis of $M(\l)$, one can define a linear map $\phi: M(\l)\ra M$
by $\phi(x^i\cdot v_{\l})=x^i\cdot v$, $i\>0$. Then it is easy to see
that $\phi$ is an $H$-module map from $M(\l)$ to $M$ satisfying $\phi(v_{\l})=v$.
Furthermore, if $M=H\cdot v$, then $\phi$ is obviously an epimorphism.
\end{proof}

Let $\langle\chi\rangle$ denote the subgroup of
$\widehat{G}$ generated by $\chi$, and $[\l]$ denote the image of $\l$ under the canonical
epimorphism $\widehat{G}\ra\widehat{G}/\langle\chi\rangle$.

Let $k[y]$ be the polynomial algebra in one variable $y$ over $k$.
When $|\chi|=s<\infty$, for any $\l \in \widehat{G}$ and monic
polynomial $f(y)\in k[y]$ with $n:={\rm deg}(f(y))\>1$,
let $N_f(\l)$ be the submodule of $M(\l)$ generated by $f(x^s)\cdot v_{\l}$, and define
$V(\l, f):=M(\l)/N_f(\l)$ to be the corresponding quotient module.
Since $f(x^s)$ is a central element of $H$, $N_f(\l)=f(x^s)\cdot M(\l)$, and hence
${\rm dim}V(\l, f)=ns$. Let $m_i$ be the image of $x^i\cdot v_{\l}$ under the canonical
epimorphism $M(\l)\ra V(\l,f)$ for all $i\>0$. Then one can easily check that
$\{m_i|0\<i\<ns-1\}$ is a $k$-basis of $V(\l, f)$.
Assume that $f(y)=y^n-\sum_{j=0}^{n-1}\a_jy^j$ with $\a_j\in k$.
Then it is easy to check that the $H$-action on $V(\l, f)$ is determined by
$$ g\cdot m_i=\chi^i(g)\l(g)m_i,\ \ x\cdot m_i=\left\{\begin{array}{ll}
m_{i+1},&0\<i\<ns-2\\
\sum_{j=0}^{n-1}\a_j m_{js},&i=ns-1\\
\end{array}\right.,$$
where $0\<i\<ns-1$, $g\in G$. Obviously, $V(\l, f)$ is a weight module and
$\Pi(V(\l, f))=\{\chi^i\l|0\<i\<s-1\}$. Moreover, $m_i=x^i\cdot m_0$ for all $0\<i\<ns-1$,
and $\{m_{js+i}|0\<j\<n-1\}$ is a $k$-basis of $V(\l, f)_{(\chi^i\l)}$ for any $0\<i\<s-1$.

If $f(y)=y-\b$ for some $\b\in k$, then $n=1$, $x\cdot m_{s-1}=\b m_0$, and hence
$x^s\cdot m=\b m$ for all $m\in V(\l, f)$. Denote $V(\l, f)$ by $V(\l, \b)$
in this case.

\begin{proposition}\prlabel{3.6}
Assume $|\chi|=s<\infty$. Let $\l \in \widehat{G}$ and let $f(y)\in k[y]$ be a monic
polynomial with $n:={\rm deg}(f(y))\>1$. Then\\
{\rm (a)} $f(x^s)\cdot m=0$ for all $m\in V(\l, f)$.\\
{\rm (b)} If $f(y)$ is irreducible and $f(y)\neq y$, then $V(\l, f)$ is simple.
\end{proposition}
\begin{proof}
(a) It follows from $N_f(\l)=f(x^s)\cdot M(\l)$ and the definition of $V(\l, f)$.

(b) Assume that $f(y)=y^n-\sum_{j=0}^{n-1}\a_jy^j$ is irreducible and $f(y)\neq y$. Then $\a_0\neq 0$.
We claim that $x\cdot m\neq 0$ for any $0\neq m\in V(\l, f)$. In fact,
if $x\cdot m=0$ for some $m\in V(\l, f)$, then $0=f(x^s)\cdot m
=(x^{(n-1)s}-\sum_{1\<j\<n-1}\a_jy^{(j-1)s})x^s\cdot m-\a_0m=-\a_0m$ by Part (a).
It follows from $\a_0\neq 0$ that $m=0$. This shows the claim.

Let $M$ be a nonzero submodule $V(\l, f)$. Since $V(\l, f)$ is a weight module,
so is $M$. Let $m$ be a nonzero weight vector of $M$. Then there is an integer $i$ with
$0\<i\<s-1$ such that $m\in M_{(\chi^i\l)}$ since
$\Pi(M)\subseteq\Pi(V(\l, f))$.
Then $x^{s-1-i}\cdot m\neq 0$ and $x^{s-1-i}\cdot m\in M_{(\chi^{s-1}\l)}$.
Replacing $m$ by $x^{s-1-i}\cdot m$, we may assume that $m\in M_{(\chi^{s-1}\l)}$.
Since $M_{(\chi^{s-1}\l)}\subseteq V(\l, f)_{(\chi^{s-1}\l)}
={\rm span}\{m_{js+s-1} | 0\<j\<n-1\}$, we have
$m=\sum_{j=0}^{n-1}\b_j m_{js+s-1}$ for some $\b_j\in k$, $0\<j\<n-1$.
Hence $m=\sum_{j=0}^{n-1}\b_j(x^{js+s-1}\cdot m_0)=(\sum_{j=0}^{n-1}\b_jx^{js})\cdot(x^{s-1}\cdot m_0)$.
Let $f_1(y)=\sum_{j=0}^{n-1}\b_jy^j$. Then $f_1(y)$ is a nonzero polynomial
with ${\rm deg}(f_1(y))<n$, and $f_1(x^s)\cdot(x^{s-1}\cdot m_0)=m$. Since $f(y)$ is irreducible and ${\rm deg}(f(y))=n$,
$f_1(y)$ and $f(y)$ are coprime in $k[y]$. Hence there are $u(y), v(y)\in k[y]$ such that
$u(y)f_1(y)+v(y)f(y)=1$, and so $u(x^s)f_1(x^s)+v(x^s)f(x^s)=1$ in $H$.
Then by Part (a), we have
$x^{s-1}\cdot m_0=(u(x^s)f_1(x^s)+v(x^s)f(x^s))\cdot(x^{s-1}\cdot m_0)=u(x^s)\cdot m$.
Hence $x^{s-1}\cdot m_0\in M$ by $m\in M$, and so $x^j\cdot m_0\in M$
for all $j\>s-1$. Now from $f(x^s)\cdot m_0=0$, one gets that
$\a_0m_0=x^{ns}\cdot m_0-\sum_{1\<j\<n-1}\a_j(x^{js}\cdot m_0)\in M$, and hence $m_0\in M$
by $\a_0\neq 0$. It follows that $M=V(\l, f)$, and consequently, $V(\l, f)$ is simple.
\end{proof}

\begin{lemma}\lelabel{3.7}
Assume $|\chi|=s<\infty$. Let $\l_j \in \widehat{G}$, and let $f_j(y)\in k[y]$ be a monic
polynomial with $n_j:={\rm deg}(f_j(y))\>1$, $j=1, 2$. Then
$V(\l_1, f_1)\cong V(\l_2, f_2)$ if and only if one of the following two conditions is satisfied:\\
{\rm (a)} $[\l_1]=[\l_2]$ and $f_1(y)=f_2(y)$ with nonzero constant term when $\l_1\neq \l_2$.\\
{\rm (b)} $\l_1=\l_2$ and $f_1(y)=f_2(y)$.
\end{lemma}
\begin{proof}
Let $\{m_0, m_1, \cdots, m_{n_1s-1}\}$ and $\{v_0, v_1, \cdots, v_{n_2s-1}\}$ be the bases of
$V(\l_1, f_1)$ and $V(\l_2, f_2)$ as given before, respectively.

Assume that $\phi: V(\l_1, f_1)\ra V(\l_2, f_2)$ is an isomorphism of $H$-modules. Then
$\Pi(V(\l_1, f_1))=\Pi(V(\l_2, f_2))$ and
${\rm dim}V(\l_1, f_1)={\rm dim}V(\l_2, f_2)$. Hence $[\l_1]=[\l_2]$ in $\widehat{G}/\langle\chi\rangle$
and $n_1=n_2$. Put $n:=n_1=n_2$. Then we may assume that $f_1(y)=y^n-\sum_{j=0}^{n-1}\a_jy^j$
and $f_2(y)=y^n-\sum_{j=0}^{n-1}\b_jy^j$ for some $\a_j, \b_j\in k$.
Since $\phi: V(\l_1, f_1)\ra V(\l_2, f_2)$ is an $H$-module isomorphism,
there exists a unique element $m\in V(\l_1, f_1)$ such that $\phi(m)=v_0$.
Note that $v_j=x^j\cdot v_0$ for all $0\<j\<ns-1$.
By \prref{3.6}(a), we have $\phi(x^{ns}\cdot m)=\phi(\sum_{j=0}^{n-1} \a_jx^{js}\cdot m)
=\sum_{j=0}^{n-1} \a_jx^{js}\cdot\phi(m)=\sum_{j=0}^{n-1} \a_jx^{js}\cdot v_0
=\sum_{j=0}^{n-1} \a_jv_{js}$ and $\phi(x^{ns}\cdot m)=x^{ns}\cdot\phi(m)=x^{ns}\cdot v_0
=\sum_{j=0}^{n-1} \b_jx^{js}\cdot v_0=\sum_{j=0}^{n-1} \b_jv_{js}$. It follows that
$f_1(y)=f_2(y)$.

If $\l_1=\l_2$ then the condition (b) is satisfied. Now suppose that $\l_1\neq\l_2$.
Then $f_1(y)=f_2(y)=y^n-\sum_{j=0}^{n-1}\b_jy^j$.
Since $[\l_1]=[\l_2]$ as shown above, $\l_1=\chi^i\l_2$ for some $1\<i\<s-1$.
Let $u_0=\phi(m_0)$. Then $u_0\in V(\l_2, f_2)_{(\chi^i\l_2)}$ by $m_0\in V(\l_1, f_1)_{(\l_1)}$.
Thus, one can check that ${\rm Im}(\phi)={\rm span}\{x^t\cdot u_0 | 0\<t\<ns-1\}$.
Let $N={\rm span}\{v_t|i\<t\<ns-1\}$.
Then $N$ is a subspace of $V(\l_2, f_2)$ and $N\neq V(\l_2, f_2)$.
Obviously, $u_0\in N$. If $\b_0=0$ then one can check that ${\rm Im}(\phi)\subseteq N$.
This is impossible since $\phi$ is an isomorphism.
Therefore, $\b_0\neq 0$ when $\l_1\neq \l_2$, and so the condition (a) is satisfied in this case.

Conversely, if $\l_1=\l_2$ and $f_1(y)=f_2(y)$, then obviously $V(\l_1, f_1)\cong V(\l_2, f_2)$.
Now assume that $f_1(y)=f_2(y)$ with nonzero constant term and $\l_1=\chi^i\l_2$
for some $1\<i\<s-1$.
Let $f(y)=f_1(y)=f_2(y)$. Then $n:={\rm deg}(f(y))\>1$ and
$f(y)=y^n-\sum_{j=0}^{n-1}\a_jy^j$ for some $\a_j\in k$ with $\a_0\neq 0$.
For any $0\<t<i$, let $u_t=\sum_{j=0}^{n-1}\a_jv_{js+t}\in V(\l_2, f)_{(\chi^t\l_2)}$.
Then $v_t=\a_0^{-1}(u_t-\sum_{1\<j\<n-1}\a_jv_{js+t})$, and
$\{v_{js+t}, u_t|1\<j\<n-1\}$ is a $k$-basis of $V(\l_2, f)_{(\chi^t\l_2)}$,
where $0\<t<i$. Therefore, $\{v_j, u_t|i\<j\<ns-1, 0\<t\<i-1\}$ is a $k$-basis
of $V(\l_2, f)$. Thus, one can define a linear isomorphism
$\phi: V(\l_1, f)\ra V(\l_2, f)$ by $\phi(m_j)=v_{i+j}$ for $0\<j\<ns-1-i$
and $\phi(m_j)=u_{j-ns+i}$ for $ns-i\<j\<ns-1$. By \prref{3.6}(a),
one can check that $u_t=x^{ns+t}\cdot v_0$, $0\<t<i$, and that $\phi$
is an $H$-module map.
\end{proof}

Let ${\rm Irr}[y]$ denote the subset of $k[y]$ consisting of all monic irreducible polynomials
$f(y)$ with $f(y)\neq y$.

\begin{corollary}\colabel{3.8}
Assume $|\chi|=s<\infty$. Let $\l_j \in \widehat{G}$ and $f_j(y)\in{\rm Irr}[y]$, $j=1, 2$. Then
$V(\l_1, f_1)\cong V(\l_2, f_2)$ if and only if $[\l_1]=[\l_2]$ and $f_1(y)=f_2(y)$.
\end{corollary}
\begin{proof}
It follows from \leref{3.7} since a monic irreducible polynomial in $k[y]$ has nonzero constant term
if and only if it is not equal to $y$.
\end{proof}

\begin{theorem}\thlabel{3.9}
Let $M\in{\mathcal W}$ be simple. Then $M$ is finite dimensional. Moreover,
we have:\\
{ \rm (a)} If ${\rm dim} M=1$, then $M\cong V_{\l}$ for some $\l\in\widehat{G}$.\\
{ \rm (b)} If ${\rm dim} M>1$, then $|\chi|=s<\oo$ and $M\cong V(\l, f)$ for some $\l\in\widehat{G}$
and $f(y)\in{\rm Irr}[y]$.
\end{theorem}
\begin{proof}
Let $M$ be a simple weight $H$-module.
Then $\Pi(M)$ is a non-empty set.
Let $M_0=\{m \in M\mid x\cdot m=0\}$. Then one can easily check that $M_0$ is a submodule of $M$.
Hence, if $M_0\neq 0$, then $M=M_0$ since $M$ is simple, and so $M$ is isomorphic to some $V_{\l}$
by $\Pi(M)\neq \emptyset$.

Now assume that $M_0=0$. Pick up a $\l\in\Pi(M)$
and $0\neq v\in M_{(\l)}$. Then $M=H\cdot v$ since $M$ is simple.
By \leref{3.5}, there exists an $H$-module epimorphism $\phi: M(\l)\ra M$
such that $\phi(v_{\l})=v$. Note that
$H=\oplus_{i\>0}x^iH_0$ and $H_0\cdot v=kv$. Hence $M={\rm span}\{x^i\cdot v|i\>0\}$,
and $x^i\cdot v\neq 0$ for all $i\>0$ by $M_0=0$.
Obviously, $x^i\cdot v\in M_{(\chi^i\l)}$ for all $i\>0$.
If $|\chi|=\oo$, then ${\rm span}\{x^i\cdot v|i\>1\}$ is a non-trivial submodule of $M$.
This is impossible since $M$ is simple. Hence $|\chi|=s<\oo$, and $s>1$ by $\chi^{-1}(a)\neq 1$.
If $\{x^i\cdot v|i\>0\}$ is a linearly independent set, then it is a $k$-basis of $M$.
In this case, ${\rm span}\{x^i\cdot v|i\>1\}$ is a non-trivial submodule of $M$,
which is impossible as above. Thus, there exists a positive integer $l$ minimal with respect to
$x^l\cdot v\in{\rm span}\{x^i\cdot v|0\<i\<l-1\}$. Then one can see that
$\{x^i\cdot v|0\<i\<l-1\}$ is a basis of $M$.
Hence dim$ M=l<\oo$.
If $l<s$, then $M=\oplus_{0\<i\<l-1}M_{(\chi^i\l)}$. So $\Pi(M)=\{\chi^i\l|0\<i\<l-1\}$
and $\chi^l\l\notin\Pi(M)$ by $|\chi|=s>l$. This is impossible since $0\neq x^l\cdot v\in M_{(\chi^l\l)}$.
Therefore, $l\>s$ and $M=\oplus_{i=0}^{s-1}M_{(\chi^i\l)}$.
If $s\nmid l$, then $l=ns+j$ with $n\>1$ and $1\<j\<s-1$.
In this case, $x^l\cdot v\in M_{(\chi^j\l)}={\rm span}\{x^{j+ts}\cdot v|0\<t\<n-1\}$,
and hence $x^l\cdot v=\sum_{t=0}^{n-1}\a_tx^{j+ts}\cdot v$ for some $\a_t\in k$.
Let $u=x^{l-1}\cdot v-\sum_{t=0}^{n-1}\a_tx^{j-1+ts}\cdot v$. Then $u\neq 0$ by the
minimality of $l$, but $x\cdot u=0$, a contradiction since $M_0=0$. Hence $s|l$, and so
$l=ns$ with $n\>1$. Moreover, we have
$x^{ns}\cdot v=\sum_{t=0}^{n-1}\a_tx^{ts}\cdot v$ for some $\a_t\in k$ as shown above.
Let $f(y)=y^n-\sum_{t=0}^{n-1}\a_ty^t$ in $k[y]$. Then $f(x^s)\cdot v=0$, and so $f(x^s)\cdot M=0$.
Hence $N_f(\l)\subseteq{\rm Ker}(\phi)$, and $\phi$ induces an $H$-module epimorphism
$\ol{\phi}: V(\l, f)\ra M$.
Suppose that $f(y)=f_1(y)f_2(y)$ for some monic polynomials $f_i(y)\in k[y]$
with $n_i:={\rm deg}(f_i(y))\>1$, $i=1, 2$. Then $n_i<n$, and so $n_is<ns=l$, $i=1, 2$.
If $f_i(x^s)\cdot v=0$, then $x^{n_is}\cdot v\in{\rm span}\{x^j\cdot v|0\<j\<n_is-1\}$,
which is impossible by the minimality of $l$. Hence $f_i(x^s)\cdot v\neq 0$, $i=1, 2$.
Let $v_0=f_2(x^s)\cdot v$. Then $0\neq v_0\in M_{(\l)}$ and $f_1(x^s)\cdot v_0=0$.
This shows that $x^{n_1s}\cdot v_0\in{\rm span}\{x^j\cdot v_0|0\<j\<n_1s-1\}$.
Consequently, $N:={\rm span}\{x^j\cdot v_0|0\<j\<n_1s-1\}$ is a nonzero submodule
of $M$ and ${\rm dim} N\<n_1s<ns=l={\rm dim} M$, which is impossible since $M$ is simple.
Therefore, $f(y)$ is an irreducible polynomial in $k[y]$. Obviously, $f(y)\neq y$.
It follows from \prref{3.6}(b) that $V(\l, f)$ is simple, and so $\ol{\phi}$
is an isomorphism. This completes the proof.
\end{proof}

\begin{corollary}\colabel{3.10}
Assume $|\chi|=s<\infty$. Let $\l \in \widehat{G}$ and let $f(y)\in k[y]$ be a monic
polynomial with $n:={\rm deg}(f(y))\>1$. Then $V(\l, f)$ is simple if and only if
$f(y)\in{\rm Irr}[y]$.
\end{corollary}
\begin{proof}
It follows from \prref{3.6}(b), \leref{3.7} and \thref{3.9}(b).
\end{proof}

\begin{corollary}\colabel{3.11}
{ \rm (a)} If $|\chi|=\oo$, then $\{V_{\l}\mid \l\in \widehat{G}\}$ is a representative set
of isomorphic classes of simple objects in $\mathcal W$.\\
{\rm (b)} If $|\chi|=s<\oo$, then $\{V_{\l}, V(\s, f)\mid \l\in\widehat{G},
[\s]\in\widehat{G}/\langle\chi\rangle, f(y)\in{\rm Irr}[y]\}$
is a representative set of isomorphic classes of simple objects in $\mathcal W$.\\
\end{corollary}
\begin{proof}
It follows from \leref{3.3}, \prref{3.6}(b), \coref{3.8} and \thref{3.9}.
\end{proof}

If $k$ is an algebraically closed field, then a polynomial in $k[y]$ is irreducible
if and only if it is of degree one. Hence ${\rm Irr}[y]=\{y-\b|\b\in k^{\times}\}$
in this case. Thus, by \coref{3.11}, we have the following corollary.

\begin{corollary}\colabel{3.12}
Assume that $k$ is an algebraically closed field. Then\\
{ \rm (a)} If $|\chi|=\oo$, then $\{V_{\l}\mid \l\in \widehat{G}\}$ is a representative set
of isomorphic classes of simple objects in $\mathcal W$.\\
{\rm (b)} If $|\chi|=s<\oo$, then $\{V_{\l}, V(\s, \b)\mid \l\in\widehat{G},
[\s]\in\widehat{G}/\langle\chi\rangle, \b\in k^{\times}\}$
is a representative set of isomorphic classes of simple objects in $\mathcal W$.
\end{corollary}

Assume that  $k$ is an algebraically closed field and $G$ is an abelian group.
Let $M$ be a finite dimensional simple $H$-module.
Then $M$ is a weight module. In fact, since $k$ is an
algebraically closed field and $kG$ is a commutative algebra, there is a $\l \in \widehat{G}$ such that
$M_{(\l)}\neq 0$. Hence $\op _{\l\in \widehat{G}}M_{(\l)}=\op _{\l\in \Pi(M)}M_{(\l)}$
is a nonzero submodule of $M$. Hence $M=\op _{\l\in \Pi(M)}M_{(\l)}$
since $M$ is simple, which implies that $M$ is a weight module.
When char$(k)=0$ and $|\chi|=\oo$, this statement can be described by using the results of
Andruskiewitsch, Radford and Schneider, see \cite{AndRadSch}.

\begin{corollary}\colabel{3.13}
Assume that $k$ is an algebraically close field and $G$ is an abelian group.
Let $M$ be a simple $H$-module.
Then $M$ is finite dimensional if and only if $M$ is a weight module. In this case,
$M\cong V_{\l}$ for some $\l\in\widehat{G}$, or $M\cong V(\l, \b)$ with $|\chi|=s<\oo$
for some $\l\in\widehat{G}$ and $\b\in k^{\times}$.
\end{corollary}

\begin{remark}\relabel{3.14}
Assume that $k$ is an algebraically closed field. In case that $G$ is a finite abelian group,
any simple $H$-module is a weight module, and hence
\coref{3.13} describes all simple $H$-modules. In case that $G$ is a finite (or simple) non-abelian
group, a simple $H_0$-module is not necessarily one dimensional, and hence
the corresponding Verma module is not necessarily generated by a
single element. In this case, the study of simple $H$-modules
is different with the abelian case. One can see \cite{Igl09} and \cite{PoVay} for
the representations of some pointed Hopf algebras with non-abelian coradicals and
the representations of non-pointed Hopf algebras, respectively.
\end{remark}

By \prref{3.1}, a $k$-space $M$ is an $H'$-module
(a simple $H'$-module) if and only if $M$ is an $H$-module (a simple $H$-module)
such that $I\cdot M=0$, where $I$ is the Hopf ideal of
$H$ as described in \thref{2.7}. Note that $\chi^{-1}(a)\neq1$, and hence $|\chi|>1$. It follows from \thref{2.7}
that $I=0$ if and only if $\chi^{-1}(a)$ is not a primitive $n$-th root of unity for any
$n\>2$. Moreover, $|\chi|=\infty$ if $I=0$. If $\chi^{-1}(a)$ is a primitive $n$-th root of unity for some
$n\>2$, then by \coref{2.9}, one knows that $I=\langle x^n\rangle$,
or $I=\langle x^n-\b(1-a^n)\rangle$ with $\b\neq 0$, $a^n\neq1$ and $|\chi|=n$.
Thus, by \coref{3.11}, one can get all the simple weight modules over $H'$.

\begin{corollary}\colabel{3.15}
If $\chi^{-1}(a)$ is not a primitive $n$-th root of unity for any
$n\>2$, then $\{V_{\l}\mid \l\in \widehat{G}\}$ is a representative set of
isomorphic classes of simple objects in $\mathcal{W}'$.
\end{corollary}
\begin{proof}
It follows from \coref{3.11} since $|\chi|=\oo$ in
this case.
\end{proof}

\begin{corollary}\colabel{3.16}
Assume that $\chi^{-1}(a)$ is a primitive $n$-th root of unity for some $n\>2$.
Then a representative set $\mathcal S$ of isomorphic classes of simple objects in $\mathcal{W}'$
can be described as follows:\\
{ \rm (a)} If $I=\langle x^n \rangle$, then $\mathcal S=\{V_{\l}\mid \l\in \widehat{G}\}$.\\
{ \rm (b)} If $I=\langle x^n-\b(1-a^n)\rangle$ with $\b\in k^{\ti}$, $a^n\neq 1$ and $|\chi|=n$,
then
$$\mathcal S=\left\{V_{\l}, V({\sigma}, \b(1-\sigma(a)^n))\mid \l\in \widehat{G} \text{ with }
 \l(a)^n=1, [\sigma]\in \widehat{G}/\langle\chi\rangle \text{ with }\sigma (a)^n\neq1\right\}.$$
\end{corollary}
\begin{proof}
Part (a) can be shown easily. Now we show Part (b).
Assume $I=\langle x^n-\b(1-a^n)\rangle$ with $\b\in k^{\ti}$, $a^n\neq 1$ and $|\chi|=n$.
Let $M$ be a simple weight $H'$-module. Then $M$ is a simple weight $H$-module and $(x^n-\b(1-a^n))\cdot M=0$.
By \thref{3.9}, ${\rm dim}M=t<\oo$.

We first consider the case of $t=1$. In this case, by the proof of
\thref{3.9}, one knows that $x\cdot M=0$ and $M\cong V_{\l}$ for some $\l\in\widehat{G}$.
Then from $(x^n-\b(1-a^n))\cdot M=0$ and $\b\neq 0$, one gets $\l(a)^n=1$.
Next, we consider the case of $t>1$. In this case, $n|t$ by $|\chi|=n<\oo$.
Let $\s\in\Pi(M)$ and $0\neq v\in M_{(\s)}$. Then by the proof of
\thref{3.9}, $t$ is the smallest positive integer such that $x^t\cdot v\in{\rm span}\{x^i\cdot v|0\<i\<t-1\}$.
Moreover, $M={\rm span}\{x^i\cdot v|0\<i\<t-1\}$.
Since $(x^n-\b(1-a^n))\cdot M=0$, $x^n\cdot v=\b(1-a^n)\cdot v=\b(1-\s(a)^n)v\in {\rm span}\{x^i\cdot v|0\<i\<n-1\}$.
Hence $t=n$ by $n\<t$ and the minimality of $t$. Thus, $M={\rm span}\{x^i\cdot v|0\<i\<n-1\}\cong V(\s, \b(1-\s(a)^n))$.
Since $M$ is simple, $\b(1-\s(a)^n)\neq 0$. Hence $\s(a)^n\neq 1$.
Furthermore, if $[\l]=[\s]$, then $\l(a)^n=\s(a)^n$ by $|\chi|=n$ or the assumption that
$\chi^{-1}(a)$ is a primitive $n$-th root of unity. It follows from \coref{3.8} that
$V({\l},\b(1-\l(a)^n))\cong V({\sigma},\b(1-\sigma(a)^n))$
if and only if $[\l]=[\sigma]$.

Finally, for any $\l\in \widehat{G}$ with $\l(a)^n=1$ and $\sigma\in \widehat{G}$ with $\sigma (a)^n\neq1$,
one can see that
$(x^n-\b(1-a^n))\cdot V_{\l}=0$ and $(x^n-\b(1-a^n))\cdot V({\sigma}, \b(1-\sigma(a)^n))=0$.
Therefore, $V_{\l}$ and $V({\sigma}, \b(1-\sigma(a)^n))$ are $H'$-modules in this case.
This completes the proof.
\end{proof}

In the rest of this section, we consider some tensor products of simple weight modules over $H$ or $H'$.
By Corollaries 3.15 and 3.16, 
one knows that there is a simple weight $H'$-module with dimension $>1$
only if $|\chi|=s<\oo$ and
$\chi^{-1}(a)$ is a primitive $s$-th root of unity.

\begin{proposition}\prlabel{3.17}
Let $\sigma ,\l \in \widehat{G}$ and $\a\in k$. Then\\
{\rm (a)} $V_{\l}\otimes V_{\s}\cong V_{\l\s}$.\\
{\rm (b)} Assume $|\chi|=s<\infty$. Then
$V_{\l}\ot V(\sigma, \a)\cong V(\l\sigma, \a)$ and $V(\sigma, \a)\ot V_{\l}\cong V(\sigma\l, \a\l(a)^s)$.
\end{proposition}
\begin{proof}
It follows from a straightforward verification.
\end{proof}

\begin{proposition}\prlabel{3.18}
Assume that $|\chi|=s<\infty$ and that $\chi^{-1}(a)$ is a primitive $s$-th root of unity.
Let $\sigma ,\l \in \widehat{G}$ and $\a, \b\in k$. Then
$V(\sigma, \a)\ot V(\l, \b)\cong\oplus_{t=0}^{s-1}V(\chi^t\s\l, \a\l(a)^s+\b)$.
\end{proposition}
\begin{proof}
Let $M=V(\sigma, \a)\ot V(\l, \b)$. Let $\{m_0, m_1, \cdots, m_{s-1}\}$ and $\{v_0, v_1, \cdots, v_{s-1}\}$
be the bases of $V(\sigma, \a)$ and $V(\l, \b)$ as described before, respectively.
Then $\{m_i\ot v_j|0\<i, j\<s-1\}$ is a $k$-basis of $M$. For any $0\<i, j\<s-1$ and $g\in G$, we have
$$g\cdot(m_i\ot v_j)=g\cdot m_i\ot g\cdot v_j=(\chi^{i+j}\sigma\l)(g)m_i\ot v_j,$$
and hence $m_i\ot v_j\in M_{(\chi^{i+j}\sigma\l)}$. Thus, $M$ is a weight module. Since $|\chi|=s$,
$M=M_{(\sigma\l)}\oplus M_{(\chi\sigma\l)}\oplus\cdots\oplus M_{(\chi^{s-1}\sigma\l)}$, and
for any $0\<t\<s-1$,
$$M_{(\chi^t\sigma\l)}={\rm span}\{m_i\ot v_{t-i}, m_j\ot v_{s+t-j} \mid 0\<i\< t, t+1\<j\<s-1\}.$$
Thus, dim$(M_{(\chi^t\sigma\l)})=s$, and moreover,
$x\cdot M_{(\chi^t\sigma\l)}\subseteq M_{(\chi^{t+1}\sigma\l)}$ for all $0\<t\<s-1$.
Since $q:=\chi^{-1}(a)$ is a primitive $s$-th root of unity
and $(1\ot x)(x\ot a)=q(x\ot a)(1\ot x)$, we have
$\Delta(x^s)=x^s\ot a^s+1\ot x^s$. Now for any $0\<i, j\<s-1$, a straightforward computation shows that
$x^s\cdot(m_i\ot v_j)=(\a\l(a)^s+\b)m_i\ot v_j$.
Hence $x^s\cdot m=(\a\l(a)^s+\b)m$ for all $m\in M$.

(a) Assume $\a\l(a)^s+\b\neq 0$. Then $x\cdot: M\ra M$, $m\mapsto x\cdot m$, is a linear
automorphism of $M$. Let $\{u_1, u_2, \cdots, u_s\}$ be a basis of $M_{(\sigma\l)}$. Then
$\{x^t\cdot u_1, x^t\cdot u_2, \cdots, x^t\cdot u_s\}$ is a basis of $M_{(\chi^t\sigma\l)}$
for all $0\<t\<s-1$. It follows that $\{x^t\cdot u_i\mid 0\<t\<s-1, 1\<i\<s\}$ is a basis of
$M$. Let $U_i={\rm span}\{x^t\cdot u_i|0\<t\<s-1\}$ for any $1\<i\<s$. Then
$M=U_1\oplus U_2\oplus\cdots\oplus U_s$. It is easy to see that $U_i$ is a submodule of $M$
and $U_i\cong V(\sigma\l, \a\l(a)^s+\b)$ for all $1\<i\<s$.
By \coref{3.8}, $V(\chi^t\sigma\l, \a\l(a)^s+\b)\cong V(\sigma\l, \a\l(a)^s+\b)$
for all $0\<t\<s-1$. It follows that $V(\sigma, \a)\ot V(\l, \b)\cong\oplus_{t=0}^{s-1}V(\chi^t\s\l, \a\l(a)^s+\b)$
in this case.

(b) Assume $\a\l(a)^s+\b=0$. Let $f: M\ra M$ be the linear endomorphism of $M$ defined by
$f(m)=x\cdot m$, $m\in M$. Then $f^s=0$. Let $m\in M_{(\sigma\l)}$. Then
$$m=\a_0 m_0\ot v_0+\a_1m_1\ot v_{s-1}+\a_2m_2\ot v_{s-2}+\cdots+\a_{s-1}m_{s-1}\ot v_1$$
for some $\a_0, \a_1, \cdots, \a_{s-1}\in k$. By a straightforward computation, one gets that
$$\begin{array}{rcl}
x\cdot m&=&(\a_0+\a\l(a)q^{s-1}\a_{s-1})m_0\ot v_1+(\l(a)\a_0+\b\a_1)m_1\ot v_0+\\
&&+\sum\limits_{2\<i\<s-1}(q^{i-1}\l(a)\a_{i-1}+\a_i)m_i\ot v_{s+1-i}.\\
\end{array}$$
Thus, $x\cdot m=0$ if and only if $(\a_0, \a_1, \cdots, \a_{s-1})$ is a solution of the following
system of linear equations
$$\left\{\begin{array}{ccc}
           x_0+\a\l(a)q^{s-1}x_{s-1} & = & 0 \\
           \l(a)x_0+\b x_1 & = & 0 \\
           \l(a)q x_1+x_2  & = & 0 \\
           \l(a)q^2x_2+x_3 & = & 0 \\
           \cdots\cdots &  &  \\
           \l(a)q^{s-2}x_{s-2}+x_{s-1} & = & 0
         \end{array}\right.$$
Since $q$ is a primitive $s$-th root of unity and $\a\l(a)^s+\b=0$, one can easily check that
the rank of the coefficient matrix of the above system of linear equations is $s-1$.
It follows that ${\rm dim(Ker}(f)\cap M_{(\sigma\l)})=1$.
Similarly, one can show that ${\rm dim}({\rm Ker}(f)\cap M_{(\chi^t\sigma\l)})=1$
for all $1\<t\<s-1$. Since $x\cdot M_{(\chi^t\sigma\l)}\subseteq M_{(\chi^{t+1}\sigma\l)}$, we have
$${\rm Ker}(f)=({\rm Ker}(f)\cap M_{(\sigma\l)})\oplus({\rm Ker}(f)\cap M_{(\chi\sigma\l)})\oplus
\cdots\oplus({\rm Ker}(f)\cap M_{(\chi^{s-1}\sigma\l)})$$
and so ${\rm dim}({\rm Ker}(f))=s$.
Thus, ${\rm dim}({\rm Im}(f^i))={\rm dim}({\rm Im}(f^{i-1}))-{\rm dim}({\rm Im}(f^{i-1})\cap{\rm Ker}(f))
\>{\rm dim}({\rm Im}(f^{i-1}))-s$, which implies that ${\rm dim}({\rm Im}(f^i))\>{\rm dim}({\rm Im}(f^0))-is=s^2-is=(s-i)s$,
where $i\>1$. In particular, ${\rm dim}({\rm Im}(f^{s-1}))\>s$. Since $f^s=0$,
${\rm Im}(f^{s-1})\subseteq{\rm Ker}(f)$, and so ${\rm dim}({\rm Im}(f^{s-1}))\<{\rm dim}({\rm Ker}(f))=s$.
It follows that ${\rm dim}({\rm Im}(f^{s-1}))={\rm dim}({\rm Ker}(f))=s$. Consequently,
${\rm Im}(f^{s-1})={\rm Ker}(f)$, and
${\rm dim}({\rm Im}(f^{s-1})\cap M_{(\chi^t\s\l)})={\rm dim}({\rm Ker}(f)\cap M_{(\chi^t\s\l)})=1$
for any $0\<t\<s-1$. Since ${\rm Im}(f^{s-1})\cap M_{(\chi^t\s\l)}
\subseteq f^{s-1}(M_{(\chi^{t+1}\s\l)})$,
it follows that $f^{s-1}(M_{(\chi^t\s\l)})\neq 0$, where $0\<t\<s-1$.
Thus, for any $0\<t\<s-1$, one can choose an element $u_t\in M_{(\chi^t\s\l)}$ such that $f^{s-1}(u_t)\neq 0$.
Now let $0\<t\<s-1$. Then $u_t, f(u_{t-1}), \cdots, f^t(u_0), f^{t+1}(u_{s-1}), \cdots, f^{s-1}(u_{t+1})\in M_{(\chi^t\s\l)}$.
From $f^{s-1}(u_i)\neq 0$ and $f^s(u_i)=0$ for all $0\<i\<s-1$, one knows that the following elements
$$u_t, f(u_{t-1}), \cdots, f^t(u_0), f^{t+1}(u_{s-1}), \cdots, f^{s-1}(u_{t+1})$$
are linearly independent over $k$, and hence form  a $k$-basis of $ M_{(\chi^t\s\l)}$.
Consequently, $\{f^i(u_t)\mid 0\<i, t\<s-1\}$ is a $k$-basis of $M$.
Let $U_t={\rm span}\{f^i(u_t)\mid 0\<i\< s-1\}$ for any $0\<t\< s-1$. Then $M=U_0\oplus U_1\oplus\cdots\oplus U_{s-1}$.
One can easily check that $U_t$ is a submodule of $M$ and $U_t\cong V(\chi^t\s\l, 0)$ for any $0\<t\< s-1$.
This completes the proof.
\end{proof}

\begin{remark}
From the next section, one will see that $V(\l, \b)$ is indecomposable for any $\l\in\widehat{G}$
and $\b\in k$. Then from Corollaries 3.15, 3.16, Propositions 3.17 and 3.18,
one gets the decomposition of the tensor product of any two simple weight $H'$-modules
into the direct sum of indecomposable modules.
\end{remark}

\section{Indecomposable weight modules}\selabel{4}

In this section, we still assume that $\chi^{-1}(a)\neq 1$, and use the notations of last section.
In order to investigate the finite dimensional indecomposable weight modules over $H$ and $H'$,
we first recall some basic concepts and prove a technical lemma.

Let $A$ be a $k$-algebra, and $M$ a finite dimensional module over $A$.
The {\it radical} rad$(M)$ of $M$
is defined to be the intersection of all maximal submodules of $M$. Then one can define
rad$^i(M)$ for all $i\>0$ recursively by ${\rm rad}^0(M)=M$ and
${\rm rad}^i(M)={\rm rad}({\rm rad}^{i-1}(M))$ for $i>0$. The smallest nonnegative integer $i$
with ${\rm rad}^i(M)=0$ is called the {\it radical length} of $M$, denoted by ${\rm rl}(M)$,
and $0\subset{\rm rad}^{i-1}(M)\subset\cdots\subset{\rm rad}^2(M)\subset{\rm rad}(M)\subset M$
is called the {\it radical series} of $M$. Sometimes, the radical length of $M$ is called
the {\it Loewy length} of $M$. The {\it socle} soc$(M)$ of $M$ is defined to be the sum of all simple
submodules of $M$. Let soc$^0(M)=0$. For $j>0$, let ${\rm soc}^j(M)$ be the preimage
of ${\rm soc}(M/{\rm soc}^{j-1}(M))$ in $M$. The smallest nonnegative integer $t$ with soc$^t(M)=M$ is
called the {\it socle length} of $M$, denoted by sl$(M)$, and $0\subset{\rm soc}(M)\subset{\rm soc}^2(M)\subset
\cdots\subset{\rm soc}^{t-1}(M)\subset{\rm soc}^t(M)=M$ is called the {\it socle series} of $M$.
It is well-known that ${\rm rl}(M)={\rm sl}(M)$ (see \cite[Proposition II.4.7]{ARS}).
The {\it length} of $M$ is denoted by l$(M)$.

\begin{lemma}\lelabel{4.1}
Let $A$ be a $k$-algebra and $M$ a finite dimensional $A$-module with ${\rm rl}(M)=t>1$.
Assume that there is a linear endomorphism $\phi$ of $M$ such that
${\rm Ker}(\phi)={\rm soc}(M)$ and such that
$\phi(N)={\rm rad}(N)$ and $\phi^{-1}(N)$ is a submodule of $M$ for any submodule $N$ of $M$.
Then\\
$(a)$  ${\rm l(rad}^i(M)/{\rm rad}^{i+1}(M))\geq{\rm l(rad}^{i+1}(M)/{\rm rad}^{i+2}(M))$
for all $0\leqslant i\leqslant t-2$.\\
$(b)$ ${\rm l}(M/{\rm rad}(M))={\rm l}({\rm soc}(M))$.\\
$(c)$ If $M$ is indecomposable, then
${\rm l}(M/{\rm rad}(M))={\rm l}({\rm rad}^{t-1}(M))=1$, and $M$ is uniserial.
\end{lemma}

\begin{proof}
(a) Let $0\<i\<t-2$. Then $\phi({\rm rad}^i(M))={\rm rad}^{i+1}(M)$.
By the restriction of $\phi$ on ${\rm rad}^i(M)$, one gets a linear
epimorphism $\phi_i :{\rm rad}^i(M)\ra{\rm rad}^{i+1}(M)$. Moreover,
$\phi_i({\rm rad}^{i+1}(M))={\rm rad}^{i+2}(M)$, and hence $\phi_i$ induces a linear
epimorphism
$$\psi_i: {\rm rad}^i(M)/{\rm rad}^{i+1}(M)\ra
{\rm rad}^{i+1}(M)/{\rm rad}^{i+2}(M), \
v+{\rm rad}^{i+1}(M)\mapsto \phi_i(x)+{\rm rad}^{i+2}(M).$$
Let $V$ be a simple submodule of ${\rm rad}^{i+1}(M)/{\rm rad}^{i+2}(M)$.
Then one can check that $\psi_i^{-1}(V)$ is a submodule of ${\rm rad}^{i}(M)/{\rm rad}^{i+1}(M)$.
Since ${\rm rad}^i(M)/{\rm rad}^{i+1}(M)$ is semisimple,
$\psi_i^{-1}(V)=\sum_{j=1}^nW_i$ for some simple submodules
$W_1, W_2, \cdots, W_n$ of ${\rm rad}^i(M)/{\rm rad}^{i+1}(M)$.
Then for any $1\<j\<n$, $\psi_i(W_j)$ is a submodule of ${\rm rad}^{i+1}(M)/{\rm rad}^{i+2}(M)$
and $\psi_i(W_j)\subseteq V$, and so $\psi_i(W_j)=0$ or
$\psi_i(W_j)=V$ since $V$ is simple. However, $V=\psi_i(\psi_i^{-1}(V))
=\sum_{j=1}^n\psi_i(W_j)$, which implies that $\psi_i(W_j)=V$ for some $j$.

Since ${\rm rad}^{i+1}(M)/{\rm rad}^{i+2}(M)$ is semisimple,
${\rm rad}^{i+1}(M)/{\rm rad}^{i+2}(M)=\oplus_{j=1}^lV_j$ for some simple
submodules $V_1, V_2, \cdots, V_l$ of ${\rm rad}^{i+1}(M)/{\rm rad}^{i+2}(M)$.
By the above discussion, there are simple submodules $U_1, U_2, \cdots, U_l$
of ${\rm rad}^i(M)/{\rm rad}^{i+1}(M)$ such that $\psi_i(U_j)=V_j$, $1\<j\<l$.
Obviously, the sum $U_1+U_2+\cdots+U_l$ is direct. It follows that
${\rm l(rad}^i(M)/{\rm rad}^{i+1}(M))\>{\rm l(rad}^{i+1}(M)/{\rm rad}^{i+2}(M))$.

(b) Since $\phi(M)={\rm rad}(M)$, $\phi$ can be regarded as a linear epimorphism
from $M$ to ${\rm rad}(M)$. Let $0=U_0\subset U_1\subset U_2\subset\cdots\subset U_n={\rm rad}(M)$
be a composition series of ${\rm rad}(M)$, and let $M_i=\phi^{-1}(U_i)$ for all $0\<i\<n$.
Then by the hypothesis, ${\rm soc}(M)=M_0\subset M_1\subset M_2\subset\cdots\subset M_n=M$
is an ascending chain of submodules of $M$. Let $1\<i\<n$, and let $N$ be a submodule of $M$ with
$M_{i-1}\subseteq N\subseteq M_i$. Then $\phi(N)$ is a submodule ${\rm rad}(M)$ and
$U_{i-1}\subseteq\phi(N)\subseteq U_i$, which implies that $\phi(N)=U_{i-1}$ or
$\phi(N)=U_{i}$. Since ${\rm Ker}(\phi)=M_0\subseteq N$, $\phi^{-1}(\phi(N))=N$, and so
$N=M_{i-1}$ or $N=M_i$. It follows that ${\rm l}(M/{\rm soc}(M))={\rm l}({\rm rad}(M))$,
and so  ${\rm l}(M/{\rm rad}(M))={\rm l}({\rm soc}(M))$.

(c) Assume that $M$ is indecomposable. We first show that ${\rm l}(M/{\rm rad}(M))={\rm l}({\rm rad}^{t-1}(M))$.
By contrary, suppose that ${\rm l}(M/{\rm rad}(M))\neq{\rm l}({\rm rad}^{t-1}(M))$.
Then by Part (a), there is an integer $i$ with $0\leq i\leq t-2$ such that
$$\begin{array}{rcl}
{\rm l}(M/{\rm rad}(M))&=&{\rm l}({\rm rad}(M)/{\rm rad}^2(M))=\cdots\\
&=&{\rm l}({\rm rad}^i(M)/{\rm rad}^{i+1}(M))\\
&>&{\rm l}({\rm rad}^{i+1}(M)/{\rm rad}^{i+2}(M)).\\
\end{array}$$
Let $V={\rm rad}^i(M)$. Then rl$(V)=t-i$ and
${\rm l}(V/{\rm rad}(V))>{\rm l}({\rm rad}(V)/{\rm rad}^2(V))
\geqslant{\rm l}({\rm rad}^{t-i-1}(V))$.
Since $\phi(V)={\rm rad}(V)\subset V$, one gets a linear map
$\phi|_V: V\ra V$ by the restriction of $\phi$ on $V$. Obviously,
$(V, \phi|_V)$ satisfies all the hypotheses on $(M, \phi)$ in the lemma.
Thus, by Part (b), we have that
l(soc$(V))={\rm l}(V/{\rm rad}(V))={\rm l}(M/{\rm rad}(M))={\rm l}({\rm soc}(M))$,
and hence ${\rm soc}(V)={\rm soc}(M)$. By the proof of Part (a),
the map $\psi_i: V/{\rm rad}(V)\ra {\rm rad}(V)/{\rm rad}^2(V)$,
$v+{\rm rad}(V)\mapsto \phi_i(v)+{\rm rad}^2(V)$, is surjective.
Obviously, Ker$(\psi_i)=({\rm soc}(V)+{\rm rad}(V))/{\rm rad}(V)$.
Then an argument similar to Part (b) shows that l$((V/{\rm rad}(V))/{\rm Ker}(\psi_i))
={\rm l}({\rm rad}(V)/{\rm rad}^2(V))$.
Since ${\rm l}(V/{\rm rad}(V))>{\rm l}({\rm rad}(V)/{\rm rad}^2(V))$,
${\rm Ker}(\psi_i)\neq 0$, and so ${\rm soc}(V)+{\rm rad}(V)\neq {\rm rad}(V)$.
Thus, there exists a nonzero submodule $U$ of soc$(V)$ such that
soc$(M)={\rm soc}(V)=U\oplus ({\rm soc}(V)\cap{\rm rad}(V))$,
and consequently, ${\rm soc}(V)+{\rm rad}(V)=U\oplus{\rm rad}(V)$.
Consider  the canonical module epimorphism $\pi: V\ra V/{\rm rad}(V)$.
Then $\pi(U)=(U\oplus{\rm rad}(V))/{\rm rad}(V)\cong U$.
Since $V/{\rm rad}(V)$ is semisimple, there is a submodule
$W/{\rm rad}(V)$ of $V/{\rm rad}(V)$ such that
$V/{\rm rad}(V)=\pi(U)\oplus W/{\rm rad}(V)$, where $W$ is a submodule
of $V$ with ${\rm rad}(V)\subseteq W$. It follows that $V=U\oplus W$.
Note that $W/{\rm rad}(V)\neq 0$
since ${\rm soc}(V)+{\rm rad}(V)\neq V$.

Now we consider the linear map
$$\varphi: M\ra V,\ \varphi(m)=\phi^i(m).$$
Since $\phi^i(M)={\rm rad}^i(M)=V$ and $\phi^i({\rm rad}(M))={\rm rad}(V)$,
$\varphi$ is surjective and $\varphi$ induces a linear epimorphism
$\ol{\varphi}: M/{\rm rad}(M)\ra V/{\rm rad}(V), \ m+{\rm rad}(M)\mapsto\varphi(m)+{\rm rad}(V)$.
Then an argument similar to Part (b) shows that
l$((M/{\rm rad}(M))/{\rm Ker}(\ol{\varphi}))
={\rm l}(V/{\rm rad}(V))$.
Hence ${\rm Ker}(\ol{\varphi})=0$ by l$(M/{\rm rad}(M))$=l$(V/{\rm rad}(V))$,
and so $\ol{\varphi}$ is bijective.
It follows that ${\rm Ker}(\varphi)\subseteq{\rm rad}(M)$.

Let $M_1=\varphi^{-1}(U)$ and $M_2=\varphi^{-1}(W)$.
Then $M_1$ and $M_2$ are both submodules of $M$.
Moreover, $M_1+M_2=M$ and $M_1\cap M_2={\rm Ker}(\varphi)$.
Consider the restriction of $\varphi$ on $M_1$, one gets a linear
epimorphism $\varphi|_{M_1}: M_1\ra U$, which induces a linear isomorphism
$$\varphi_1: M_1/{\rm Ker}(\varphi)\ra U, \ m+{\rm Ker}(\varphi)\mapsto\varphi(m).$$
For any simple submodule $U_1$ of $U$, one can check that
$\varphi_1^{-1}(U_1)$ is a simple submodule of $M_1/{\rm Ker}(\varphi)$.
It follows that $M_1/{\rm Ker}(\varphi)$ is semisimple since $U$ is semisimple.

Let $N_1$ be a submodule of $M_1$ minimal with respect to $\varphi(N_1)=U$.
Then $M_1=N_1+{\rm Ker}(\varphi)$, and hence $M=M_1+M_2=N_1+{\rm Ker}(\varphi)+M_2=N_1+M_2$
by ${\rm Ker}(\varphi)\subseteq M_2$.
Since $N_1/(N_1\cap {\rm Ker}(\varphi))\cong M_1/{\rm Ker}(\varphi)$,
$N_1/(N_1\cap {\rm Ker}(\varphi))$ is nonzero and semisimple, and hence
${\rm rad}(N_1)\subseteq N_1\cap{\rm Ker}(\varphi)\subsetneq N_1$.
We claim that ${\rm rad}(N_1)=N_1\cap{\rm Ker}(\varphi)$. In fact, if
${\rm rad}(N_1)\neq N_1\cap{\rm Ker}(\varphi)$, then $(N_1\cap{\rm Ker}(\varphi))/{\rm rad}(N_1)$
is a nontrivial submodule of $N_1/{\rm rad}(N_1)$. Hence
$N_1/{\rm rad}(N_1)=M'/{\rm rad}(N_1)\oplus(N_1\cap{\rm Ker}(\varphi))/{\rm rad}(N_1)$,
where $M'$ is a proper submodule of $N_1$ containing ${\rm rad}(N_1)$.
In this case, $N_1=M'+(N_1\cap{\rm Ker}(\varphi))$, and so $\varphi(M')=\varphi(N_1)=U$,
which contradicts the minimality of $N_1$. This shows the claim.
Hence
$N_1\cap M_2=N_1\cap M_1\cap M_2
=N_1\cap{\rm Ker}(\varphi)={\rm rad}(N_1)$.

Let $N_2$ be a submodule of $M_2$ minimal with respect to $\varphi(N_2)=W$.
Then $M_2=N_2+{\rm Ker}(\varphi)$. Hence $M=N_1+M_2=N_1+N_2+{\rm Ker}(\varphi)
\subseteq N_1+N_2+{\rm rad}(M)\subseteq M$, and so $M=N_1+N_2+{\rm rad}(M)$.
This implies that $M=N_1+N_2$.
Now consider the restriction $\varphi|_{N_2}: N_2\ra W$, which is a linear epimorphism.
Since $\varphi|_{N_2}({\rm rad}(N_2))=\varphi(\phi(N_2))=\phi(\varphi(N_2))=\phi(W)={\rm rad}(W)$,
$\varphi|_{N_2}$ induces a linear epimorphism
$\ol{\varphi|_{N_2}}: N_2/{\rm rad}(N_2)\ra W/{\rm rad}(W)$.
Since $N_2/{\rm rad}(N_2)$ is semisimple and ${\rm Ker}(\ol{\varphi|_{N_2}})$
is a submodule of $N_2/{\rm rad}(N_2)$,
there is a submodule $L/{\rm rad}(N_2)$ of $N_2/{\rm rad}(N_2)$ such that
$N_2/{\rm rad}(N_2)={\rm Ker}(\ol{\varphi|_{N_2}})\oplus(L/{\rm rad}(N_2))$.
In this case, $\ol{\varphi|_{N_2}}(L/{\rm rad}(N_2))=W/{\rm rad}(W)$. This means that
$\varphi(L)+{\rm rad}(W)=W$, and hence $\varphi(L)=W$. By the minimality of $N_2$, we have $L=N_2$,
and so ${\rm Ker}(\ol{\varphi|_{N_2}})=0$. This shows that
$\ol{\varphi|_{N_2}}$ is bijective.
Therefore, ${\rm Ker}(\varphi)\cap N_2\subseteq{\rm rad}(N_2)$.
Since $M_1\cap M_2={\rm Ker}(\varphi)$,
$N_1\cap N_2=N_1\cap{\rm Ker}(\varphi)\cap N_2={\rm rad}(N_1)\cap{\rm Ker}(\varphi)\cap N_2
\subseteq {\rm rad}(N_1)\cap{\rm rad}(N_2)\subseteq N_1\cap N_2$. Hence
$N_1\cap N_2={\rm rad}(N_1)\cap{\rm rad}(N_2)$.
Let $1\<j<i$ and suppose that $N_1\cap N_2={\rm rad}^t(N_1)\cap{\rm rad}^t(N_2)$ for all $1\<t\<j$.
Since $M=N_1+N_2$, we have ${\rm rad}^j(M)=\phi^j(M)=\phi^j(N_1)+\phi^j(N_2)
={\rm rad}^j(N_1)+{\rm rad}^j(N_2)$.
Let $M'={\rm rad}^j(M)$, $N_1'={\rm rad}^j(N_1)$ and $N_2'={\rm rad}^j(N_2)$.
Consider the linear epimorphism
$$\theta: M'\ra V=U\oplus W,\ \theta(m)=\phi^{i-j}(m).$$
Then $\theta(N_1')=U$ and $\theta(N_2')=W$.
Suppose that there is a proper submodule $L$ of $N_1'$ such that $\theta(L)=U$.
Note that the map $\d: N_1\ra N_1'$, $v\mapsto\phi^j(v)$ is a linear epimorphism.
Let $K=\d^{-1}(L)$. Then $K$ is a proper submodule of $N_1$ and $\d(K)=L$. Hence
$\varphi(K)=(\theta\d)(K)=\theta(L)=U$, which contradicts the minimality
of $N_1$. This shows that $N_1'$ is a submodule of $\theta^{-1}(U)$ minimal
with respect to $\theta(N_1')=U$. Similarly, one can show that $N_2'$ is a
submodule of $\theta^{-1}(W)$ minimal with respect to $\theta(N_2')=W$. Then a similar argument as above
shows that $N_1'\cap N_2'={\rm rad}(N_1')\cap{\rm rad}(N_2')$. That is,
${\rm rad}^j(N_1)\cap{\rm rad}^j(N_2)={\rm rad}^{j+1}(N_1)\cap{\rm rad}^{j+1}(N_2)$.
Therefore, $N_1\cap N_2={\rm rad}^t(N_1)\cap{\rm rad}^t(N_2)$ for all $1\<t\<i$.
In particular, $N_1\cap N_2={\rm rad}^i(N_1)\cap{\rm rad}^i(N_2)=U\cap W=0$,
and so $M=N_1\oplus N_2$, a contradiction since $M$ is indecomposable.
This shows that ${\rm l}(M/{\rm rad}(M))={\rm l}({\rm rad}^{t-1}(M))$.

Since ${\rm rad}^{t}(M)=0$, ${\rm rad}^{t-1}(M)\subseteq{\rm soc}(M)$,
and so ${\rm rad}^{t-1}(M)={\rm soc}(M)$ by Part (b).
Then from the above proof, one can see that ${\rm rad}^{t-1}(M)$ is indecomposable,
and so ${\rm l}({\rm rad}^{t-1}(M))=1$.
Thus, $M$ is uniserial by Part (a).
\end{proof}

Now we return to consider the representations of $H$.
For any $t\>1$ and $\l\in\widehat{G}$, let $J_t(\l)$ be the submodule of $M(\l)$ generated by $x^t\cdot v_{\l}$,
and define $V_t(\l):=M(\l)/J_t(\l)$ to be the corresponding quotient module.
Then $J_t(\l)=x^t\cdot M(\l)={\rm span}\{x^i\cdot v_{\l}\mid i\>t\}$. Let $m_i$ be the image of $x^i\cdot v_{\l}$
under the canonical epimorphism $M(\l)\ra V_t(\l)$ for all $i\>0$. Then $m_i=0$ for all $i\>t$, and
$V_t(\l)$ is a $t$-dimensional vector space over $k$ with a basis $\{m_0, m_1, \cdots, m_{t-1}\}$.
One can easily check that the $H$-action on $V_t(\l)$ is determined by
$$ g\cdot m_i=\chi^i(g)\l(g)m_i,\  x\cdot m_i=m_{i+1}, 0\<i\<t-2, \ x\cdot m_{t-1}=0.$$

\begin{proposition}\prlabel{4.2}
Let $\l\in \widehat{G}$. Then $V_1(\l)\cong V_{\l}$. Furthermore, if $|\chi|=\oo$,
then $J_1(\l)$ is the unique maximal submodule of $M(\l)$.
\end{proposition}
\begin{proof}
Since $J_1(\l)=x\cdot M(\l)$, it follows from \prref{3.4}(b) that
$J_1(\l)$ is a maximal submodule of $M(\l)$ and $V_1(\l)\cong V_{\l}$.
Now assume $|\chi|=\oo$. Let $N$ be a maximal submodule of $M(\l)$.
Then $M(\l)/N$ is a simple weight module.
It follows from \coref{3.11}(a) that $M(\l)/N$ is one dimensional, and $x\cdot(M(\l)/N)=0$.
This implies that $J_1(\l)=x\cdot M(\l)\subseteq N$. Therefore, $N=J_1(\l)$ since
both $N$ and $J_1(\l)$ are maximal submodule of $M(\l)$. It follows that
$J_1(\l)$ is the unique maximal submodule of $M(\l)$.
\end{proof}

\begin{proposition}\prlabel{4.3}
Let $\l \in \widehat{G}$ and $t\>1$. Then\\
{\rm (a)} $V_t(\l)$ is simple if and only if $t=1$.\\
{\rm (b)} $V_t(\l)$ is an indecomposable and uniserial weight module.
Moreover, ${\rm l}(V_t(\l))=t$.
\end{proposition}

\begin{proof}
Let $\{m_0, m_1, \cdots, m_{t-1}\}$ be the basis of $V_t(\l)$ given as above.
Then $V_t(\l)$ is a cyclic $H$-module generated by $m_0$.

(a) $V_1(\l)$ is simple by \prref{4.2}. If $t>1$, then
$km_{t-1}$ is a non-trivial submodule of $V_t(\l)$, and
hence $V_t(\l)$ is not simple.

(b) Obviously, $V_t(\l)$ is a weight module.
For any $0\<j\<t-1$, one can easily see that $M_j={\rm span}\{m_i\mid j\<i\<t-1\}$
is a submodule of $V_t(\l)$ of dimension $t-j$. Obviously,
$$V_t(\l)=M_0\supset M_1\supset M_2\supset\cdots\supset M_{t-1}\supset M_t=0$$
is a composition series of $V_t(\l)$ and $M_j/M_{j+1}\cong V_{\chi^j\l}$, $0\<j\<t-1$.
Hence l$(V_t(\l))=t$.

If $|\chi|\>t$, then
$V_t(\l)=\oplus_{i=0}^{t-1}V_t(\l)_{(\chi^i\l)}$ and $V_t(\l)_{(\chi^i\l)}=km_i$ for all
$0\leq i\leq t-1$. Let $M$ be a nonzero submodule of $V_t(\l)$.
Then $M$ is a direct sum of those $km_i$ which lie in $M$. Since $M\neq 0$,
$m_i\in M$ for some $0\leq i\leq t-1$. Let $i={\rm min}\{j|m_j\in M\}$.
Then obviously, $M={\rm span}\{m_j|j\geq i\}=M_i$. Hence $V_t(\l)$ is uniserial,
and so $V_t(\l)$ is indecomposable.

Now assume $|\chi|=s<t$. Note that $s>1$ in this case.
Let $M$ be a non-zero submodule of $V_t(\l)$.
Then $x^t\cdot M\subseteq x^t\cdot V_t(\l)=0$.
Hence there is an integer $i$ with $1\<i\<t$ such that $x^i\cdot M=0$ but $x^{i-1}\cdot M\neq 0$.
If $i=t$, then $M\subseteq V_t(\l)=M_0=M_{t-i}$. If $i<t$ and
$m=\a_0 m_0+\a_1 m_1+\cdots +\a_{t-1}m_{t-1}\in M$, then
$x^i\cdot m=\a_0 m_i+\a_1 m_{i+1}+\cdots +\a_{t-i-1}m_{t-1}=0$ by $x^i\cdot M=0$,
which implies that $\a_0=\a_1=\cdots=\a_{t-i-1}=0$. Hence $m=\a_{t-i} m_{t-i}+\cdots +\a_{t-1}m_{t-1}\in M_{t-i}$,
and so $M\subseteq M_{t-i}$. Thus, we have proven that $M\subseteq M_{t-i}$.
Since $x^{i-1}\cdot M\neq 0$, one can choose an element $m\in M$
such that $x^{i-1}\cdot m\neq 0$. From $M\subseteq M_{t-i}$, we have $m=\a_{t-i} m_{t-i}+\cdots +\a_{t-1}m_{t-1}$
for some $\a_{t-i}, \cdots, \a_{t-1}\in k$. From $x^{i-1}\cdot m=\a_{t-i}m_{t-1}\neq 0$, one gets that $\a_{t-i}\neq0$.
Now we have
$$\begin{array}{lcr}
m&=&\a_{t-i} m_{t-i}+\a_{t-i+1}m_{t-i+1}\cdots +\a_{t-1}m_{t-1}\\
x\cdot m&=&\a_{t-i}m_{t-i+1}\cdots +\a_{t-2}m_{t-1}\\
\cdots&\cdots&\cdots\cdots\cdots\cdots\\
x^{i-1}\cdot m&=&\a_{t-i}m_{t-1}\\
\end{array}$$
Since $\a_{t-i}\neq0$ and $m, x\cdot m, \cdots, x^{i-1}\cdot m\in M$, one knows that
$ m_{t-i}, m_{t-i+1}, \cdots, m_{t-1}\in M$, and so $M_{t-i}\subseteq M$. Thus, $M=M_{t-i}$.
It follows that $V_t(\l)$ is a uniserial $H$-module, and consequently,
$V_t(\l)$ is indecomposable.
\end{proof}

By a straightforward verification, we have the following \prref{4.4}.

\begin{proposition}\prlabel{4.4}
Let $\sigma ,\l \in \widehat{G}$ and $t, l\>1$. Then
$V_t(\sigma)\cong V_l(\l)$ if and only if
$t=l$ and $\sigma=\l$.
\end{proposition}

\begin{lemma}\lelabel{4.5}
Let $M\in\mathcal{W}$ be indecomposable. Then $[\s]=[\l]$ for any $\s, \l\in\Pi(M)$.
\end{lemma}
\begin{proof}
Since $M$ is a nonzero weight module, $\Pi(M)$ is not empty and $M=\oplus_{\l\in\Pi(M)}M_{(\l)}$.
Let $\pi: \widehat{G}\ra \widehat{G}/\langle\chi\rangle$ be the canonical
group epimorphism. For any $\l\in\Pi(M)$, one can easily check that
$M_{[\l]}:=\oplus_{\s\in\Pi(M), [\s]=[\l]}M_{(\s)}$ is a nonzero submodule of $M$.
Moreover, $M=\oplus_{[\l]\in\pi(\Pi(M))}M_{[\l]}$. This forces that $\pi(\Pi(M))$ is a single-point subset of
$\widehat{G}/\langle\chi\rangle$ since $M$ is indecomposable, and the lemma follows.
\end{proof}

\begin{lemma}\lelabel{4.6}
Let $\l\in \widehat{G}$. Then $M(\l)$ is a projective object in $\mathcal W$.
\end{lemma}
\begin{proof}
Let $f: M\ra L$ be an epimorphism and $g: M(\l)\ra L$ be a morphism in $\mathcal W$.
Then $f(M_{(\tau)})=L_{(\tau)}$ and $g(M(\l)_{(\tau)})\ss L_{(\tau)}$ for any $\tau \in \widehat{G}$.
Hence there exists a weight vector $m\in M_{(\l)}$
such that $f(m)=g(v_{\l})$. By \leref{3.5}, there is an $H$-module map $\phi: M(\l)\ra M$
such that $\phi(v_{\l})=m$. Then for any $h\in H$, $(f\phi)(h\cdot v_{\l})=h\cdot(f\phi)(v_{\l})
=h\cdot f(m)=h\cdot g(v_{\l})=g(h\cdot v_{\l})$. Since $M(\l)=H\cdot v_{\l}$, $f\phi=g$.
Hence $M(\l)$ is a projective object in $\mathcal W$.
\end{proof}

Let $M\in\mathcal W$ be finite dimensional.
If $V$ is a composition factor of $M$, then it follows from \thref{3.9}
that ${\rm dim} V=1$, or ${\rm dim} V=ns$ with $|\chi|=s<\infty$ and $n\>1$.

From now on, unless otherwise stated, all $H$-modules considered are finite dimensional.

\begin{lemma}\lelabel{4.7}
Let $M\in\mathcal W$. If each composition factor of $M$ is one dimensional,
then $x\cdot M={\rm rad}(M)$.
\end{lemma}
\begin{proof}
Since $x\cdot V_{\l}=0$ for any $\l\in\widehat{G}$, it follows from \thref{3.9} that $x\cdot(M/{\rm rad}(M))=0$.
Hence $x\cdot M\subseteq{\rm rad}(M)$. On the other hand, it is easy to see that
$x\cdot M$ is a submodule of $M$. Let $\overline{M}:=M/(x\cdot M)\in\mathcal{W}$. Since $x\cdot\overline{M}=0$,
$\overline{M}$ is isomorphic to a direct sum of some $V_{\l}$, $\l\in\widehat{G}$. Hence
$\overline{M}$ is semisimple, which implies that ${\rm rad}(M)\subseteq x\cdot M$.
Therefore, $x\cdot M={\rm rad}(M)$.
\end{proof}

\begin{theorem}\thlabel{4.8}
Let $M\in{\mathcal W}$ be indecomposable. Assume that
each composition factor of $M$ is one dimensional.
Then $M$ is isomorphic to some $V_t(\l)$,
where $t\>1$ and $\l\in\widehat{G}$.
\end{theorem}

\begin{proof}
If rl$(M)=1$, then $M$ is semisimple. Hence $M$ is simple since $M$
is indecomposable, which implies that dim$(M)=1$ since
each composition factor of $M$ is one dimensional.
Then by \thref{3.9}(a) and \prref{4.2},
$M\cong V_{\l}\cong V_1(\l)$ for some $\l\in\widehat{G}$.

Now assume that rl$(M)=t>1$. Define a linear map $\phi: M\ra M$ by
$\phi(m)=x\cdot m$ for all $m\in M$. By $x\cdot V_{\l}=0$ for all $\l\in\widehat{G}$,
it follows from \thref{3.9}(a) that ${\rm Ker}(\phi)=\{m\in M|x\cdot m=0\}={\rm soc}(M)$.
Let $N$ be a submodule of $M$. Then $\phi(N)=x\cdot N={\rm rad}(N)$ by \leref{4.7}.
It is easy to check that $\phi^{-1}(N)$ is also a submodule of $M$.
Thus, it follows from \leref{4.1}(c) that $M$ is uniserial, and hence ${\rm l}(M)=t$.
Moreover, $M/{\rm rad}(M)\cong V_{\l}$ for some $\l\in\widehat{G}$.
By \prref{4.2}, there is an $H$-module epimorphism $p: M(\l)\ra M/{\rm rad}(M)$.
Let $\pi: M\ra M/{\rm rad}(M)$ be the canonical epimorphism.
By \leref{4.6}, there is an $H$-module map $f: M(\l)\ra M$ such that $\pi f=p$.
Since $p$ is surjective, $f(M(\l))+{\rm rad}(M)=M$, which implies $f(M(\l))=M$.
Hence $f$ is an epimorphism. By $x^t\cdot M=\phi^t(M)={\rm rad}^t(M)=0$,
one gets that $J_t(\l)=x^t\cdot M(\l)\subseteq{\rm Ker}(f)$. Therefore,
$f$ induces an $H$-module epimorphism $\ol{f}: V_t(\l)\ra M$.
By \prref{4.3}(b), we have ${\rm l}(V_t(\l))=t={\rm l}(M)$. Hence $\ol{f}$ is an isomorphism,
and so $M\cong V_t(\l)$. This completes the proof.
\end{proof}

\begin{corollary}\colabel{4.9}
Assume that $|\chi|=\infty$. Let $M\in\mathcal W$ be indecomposable.
Then $M$ is isomorphic to some $V_t(\l)$,
where $t\>1$ and $\l\in\widehat{G}$. Moreover, $\{V_t(\l)|t\>1, \l\in\widehat{G}\}$
is a representative set of finite dimensional indecomposable objects in $\mathcal W$.
\end{corollary}

\begin{proof}
By \thref{3.9}, every simple weight $H$-module is one dimensional when $|\chi|=\infty$.
Thus, the first statement follows from \thref{4.8}, and the second one
follows from Propositions 4.3 and 4.4.
\end{proof}

Assume that $|\chi|=s<\oo$. Then $M(\l)=\op_{0\<i\<s-1}M(\l)_{(\chi^i\l)}$ and
$$M(\l)_{(\chi ^i \l)}={\rm span}\{x^{ns+i}\cdot v_{\l}\mid n\>0\},\quad 0\<i\<s-1.$$
Let $M$ be an arbitrary $H$-module.
For any monic polynomial $f(y)\in k[y]$, put
$$M^{(f)}=\{m\in M|f(x^s)^r\cdot m=0 \mbox{ for some integer }r>0\}.$$
Then one can easily prove the following lemma.

\begin{lemma}\lelabel{4.10}
Assume that $|\chi|=s<\oo$. Let $M$ be an arbitrary $H$-module, and $f(y)\in k[y]$
a monic polynomial. Then $M^{(f)}$ is a submodule of $M$.
\end{lemma}

\begin{theorem}\thlabel{4.11}
Assume that $|\chi|=s<\oo$. Let $M$ be an $H$-module.
Then there are distinct monic irreducible polynomials $f_1(y), f_2(y), \cdots, f_t(y)$ in $k[y]$
such that $M=M^{(f_1)}\oplus M^{(f_2)}\oplus\cdots\oplus M^{(f_t)}$.
\end{theorem}
\begin{proof}
Let $\phi: M\ra M$ be the linear endomorphism defined by $\phi(m)=x^s\cdot m$.
Since $M$ is finite dimensional, $\phi$ has a minimal polynomial $f(y)\in k[y]$.
That is, $f(y)$ is a nonzero monic polynomial with $f(\phi)=0$ such that deg$(f(y))$ is minimal
with respect to $f(\phi)=0$. Then $f(x^s)\cdot M=0$.
Now $f(y)$ has a standard decomposition
$f(y)=f_1(y)^{r_1}f_2(y)^{r_2}\cdots f_t(y)^{r_t}$, where $f_1(y), f_2(y), \cdots, f_t(y)$
are distinct monic irreducible polynomials and $r_i\>1$ for all $1\<i\<t$.
If $t=1$, then $M=M^{(f_1)}$, done. Now assume that $t>1$.
Then the polynomials $f_1(y), f_2(y), \cdots, f_t(y)$
are pairwise coprime. For any $1\<i\<t$, let
$$F_i(y):=\prod_{1\<j\<t, j\neq i}f_j(y)^{r_j}=f_1(y)^{r_1}\cdots f_{i-1}(y)^{r_{i-1}}
f_{i+1}(y)^{r_{i+1}}\cdots f_t(y)^{r_t}.$$
Then the polynomials $F_1(y), F_2(y), \cdots, F_t(y)$ are coprime. Hence there exist
polynomials $u_1(y), u_2(y), \cdots, u_t(y)$ in $k[y]$ such that
$$u_1(y)F_1(y)+u_2(y)F_2(y)+\cdots +u_t(y)F_t(y)=1.$$
Let $m\in M$. Then
$$m=1\cdot m=u_1(x^s)F_1(x^s)\cdot m+u_2(x^s)F_2(x^s)\cdot m+\cdots +u_t(x^s)F_t(x^s)\cdot m.$$
For any $1\<i\<t$, $f_i(x^s)^{r_i}\cdot(u_i(x^s)F_i(x^s)\cdot m)=u_i(x^s)\cdot(f(x^s)\cdot m)=0$,
and hence $u_i(x^s)F_i(x^s)\cdot m\in M^{(f_i)}$. This shows that
$M=M^{(f_1)}+M^{(f_2)}+\cdots+M^{(f_t)}$. Now assume
$m_1+m_2+\cdots+m_t=0$, where $m_i\in M^{(f_i)}$, $1\<i\<t$. Then one can choose an integer
$n\>1$ such that $f_i(x^s)^n\cdot m_i=0$ for all $1\<i\<t$. Hence $f_i(x^s)^{r_in}\cdot m_i=0$
by $r_i\>1$, $1\<i\<t$. For any $1\<i\neq j\<t$, one can see that
$F_j(y)^n$ is divisible by $f_i(y)^{r_in}$, which implies that $F_j(x^s)^n\cdot m_i=0$.
Hence $0=F_j(x^s)^n\cdot(m_1+m_2+\cdots+m_t)=F_j(x^s)^n\cdot m_j$, $1\<j\<t$.
On the other hand, one can see that $f_j(y)$ and $F_j(y)$ are coprime, and so
$f_j(y)^n$ and $F_j(y)^n$ are coprime, $1\<j\<t$. Therefore, there are polynomials $v_j(y), w_j(y)\in k[y]$
such that $v_j(y)f_j(y)^n+w_j(y)F_j(y)^n=1$. Thus, $m_j=(v_j(x^s)f_j(x^s)^n+w_j(x^s)F_j(x^s)^n)\cdot m_j
=(v_j(x^s)\cdot(f_j(x^s)^n\cdot m_j)+w_j(x^s)\cdot(F_j(x^s)^n\cdot m_j)=0$, $1\<j\<t$.
This completes the proof.
\end{proof}

\begin{corollary}\colabel{4.12}
Assume that $|\chi|=s<\oo$. Let $M$ be an indecomposable $H$-module.
Then there is a monic irreducible polynomial $f(y)$ in $k[y]$ such that $M=M^{(f)}$.
\end{corollary}
\begin{proof}
It follows from \leref{4.10} and \thref{4.11}.
\end{proof}

\begin{lemma}\lelabel{4.13}
Assume that $|\chi|=s<\oo$. Let $M$ be a simple $H$-module.
Then there exists a monic irreducible polynomial $f(y)$ in $k[y]$ such that
$f(x^s)\cdot M=0$.
\end{lemma}
\begin{proof}
By \coref{4.12}, there is a monic irreducible polynomial $f(y)$ in $k[y]$ such that $M=M^{(f)}$.
Let $N=\{m\in M|f(x^s)\cdot m=0\}$. Then $N$ is a submodule of $M$ since $f(x^s)$
is central in $H$. By $M=M^{(f)}$, one knows that $N\neq 0$. Since $M$ is simple,
$N=M$, and so $f(x^s)\cdot M=0$.
\end{proof}

\begin{lemma}\lelabel{4.14}
Assume that $|\chi|=s<\oo$. Let $M\in\mathcal W$ be indecomposable.
If there is an $f(y)\in{\rm Irr}[y]$ such that
$f(x^s)\cdot M=0$, then $M$ is simple and isomorphic to $V(\l, f)$ for some $\l\in\widehat{G}$.
\end{lemma}
\begin{proof}
Suppose that there is an $f(y)\in{\rm Irr}[y]$ such that
$f(x^s)\cdot M=0$. Let $f(y)=y^n-\sum_{i=0}^{n-1}\a_iy^i$. Then $\a_0\neq 0$. Let
$M_0=\{m\in M | x\cdot m=0\}$. Then an argument similar to the proof of \prref{3.6}(b)
shows that $M_0=0$.

Since $M$ is an indecomposable weight module, $\Pi(M)$ is not empty. Let $\l\in\Pi(M)$
and $0\neq v\in M_{(\l)}$.
Since $M_0=0$, we have $x^i\cdot v\neq 0$ for all $i\>1$. Hence $\chi^i\l\in\Pi(M)$
since $x^i\cdot v$ is a weight vector of weight $\chi^i\l$, $i\>1$.
Then it follows from \leref{4.5} and $|\chi|=s<\oo$ that $\Pi(M)=\{\chi^i\l | 0\<i\<s-1\}$.
For any $0\<i\<s-1$ and $0\neq m\in M_{(\chi^i\l)}$,
by \leref{3.5}, there is an $H$-module epimorphism $\phi: M(\chi^i\l)\ra H\cdot m$
such that $\phi(v_{\chi^i\l})=m$. Obviously, $f(x^s)\cdot v_{\chi^i\l}\in{\rm Ker}(\phi)$,
and so $N_f(\chi^i\l)\subseteq{\rm Ker}(\phi)$. Therefore,
$\phi$ induces an $H$-module epimorphism $\ol{\phi}: V(\chi^i\l, f)\ra H\cdot m$.
By \prref{3.6}(b), $V(\chi^i\l, f)$ is simple.
Hence $\ol{\phi}$ is an isomorphism, and so $H\cdot m\cong V(\chi^i\l, f)\cong V(\l, f)$
by \coref{3.8}. Thus, we have proved
that any weight vector of $M$ is contained in a simple submodule
isomorphic to $V(\l, f)$. This shows that $M$ is semisimple and each simple submodule
of $M$ is isomorphic to $V(\l, f)$. Hence $M$ is simple and isomorphic to
$V(\l, f)$ since $M$ is indecomposable.
\end{proof}

\begin{lemma}\lelabel{4.15}
Assume that $|\chi|=s<\oo$ and $f(y)\in{\rm Irr}[y]$.
Let $M\in\mathcal W$ be indecomposable with $M=M^{(f)}$.
Then there exists a $\l\in\widehat{G}$ such that each composition factor of $M$ is isomorphic to $V(\l, f)$.
\end{lemma}
\begin{proof}
Since $M$ is an indecomposable weight $H$-module,
$\Pi(M)$ is not empty. Let $\l\in\Pi(M)$. Since $|\chi|=s<\oo$,
$\Pi(M)\subseteq\{\chi^i\l|0\<i\<s-1\}$ by \leref{4.5}.
Let $N$ be a composition factor of $M$. Then $N=N^{(f)}$ by
$M=M^{(f)}$. By \leref{4.13} and its proof, one knows that $f(x^s)\cdot N=0$.
Then it follows from \leref{4.14} and its proof that $N\cong V(\s, f)$
for some $\s\in\Pi(N)$. Since $\Pi(N)\subseteq\Pi(M)\subseteq\{\chi^i\l|0\<i\<s-1\}$,
$N\cong V(\l, f)$ by \coref{3.8}.
\end{proof}

\begin{lemma}\lelabel{4.16}
Assume that $|\chi|=s<\oo$ and $f(y)\in{\rm Irr}[y]$.
Let $M\in\mathcal W$ satisfy $M=M^{(f)}$.
Then ${\rm rad}(M)=f(x^s)\cdot M$.
\end{lemma}
\begin{proof}
Since $M/{\rm rad}(M)$ is semisimple, it follows from the proof of \leref{4.15}
that $f(x^s)\cdot(M/{\rm rad}(M))=0$. Hence $f(x^s)\cdot M\subseteq {\rm rad}(M)$.
On the other hand, we have $f(x^s)\cdot(M/(f(x^s)\cdot M))=0$. Then by \leref{4.14},
one knows that $M/(f(x^s)\cdot M)$ is semisimple, which implies
${\rm rad}(M)\subseteq f(x^s)\cdot M$. Hence
${\rm rad}(M)=f(x^s)\cdot M$
\end{proof}

\begin{proposition}\prlabel{4.17}
Assume that $|\chi|=s<\oo$. Let $\l\in\widehat{G}$ and $f(y)\in{\rm Irr}[y]$.
Then for any positive integer $r$, $V(\l, f^r)$ is indecomposable with ${\rm rl}(V(\l, f^r))=r$.
\end{proposition}
\begin{proof}
Since $f(x^s)^r\cdot V(\l, f^r)=0$, $V(\l, f^r)=V(\l, f^r)^{(f)}$. It follows from \leref{4.16}
that ${\rm rad}(V(\l, f^r)=f(x^s)\cdot V(\l, f^r)$. However,
$$f(x^s)\cdot V(\l, f^r)=f(x^s)\cdot(M(\l)/N_{f^r}(\l))=(f(x^s)\cdot M(\l))/N_{f^r}(\l)=N_{f}(\l)/N_{f^r}(\l).$$
Hence we have
$$V(\l, f^r)/{\rm rad}(V(\l, f^r))=(M(\l)/N_{f^r}(\l))/(N_{f}(\l)/N_{f^r}(\l))\cong M(\l)/N_{f}(\l)=V(\l, f),$$
which is simple by \prref{3.6}(b). It follows that $V(\l, f^r)$ is indecomposable.
Now we have $f(x^s)^r\cdot V(\l, f^r)=0$ and $f(x^s)^{r-1}\cdot V(\l, f^r)\neq0$. It follows from \leref{4.16}
that ${\rm rl}(V(\l, f^r)=r$. This completes the proof.
\end{proof}

\begin{theorem}\thlabel{4.18}
Assume that $|\chi|=s<\oo$. Let $M\in\mathcal W$ be indecomposable.
Then $M\cong V_t(\l)$ for some $t\>1$ and $\l\in\widehat{G}$,
or $M\cong V(\l, f^r)$ for some $\l\in\widehat{G}$, $r\>1$ and $f(y)\in{\rm Irr}[y]$.
Moreover, $M$ is uniserial.
\end{theorem}
\begin{proof}
By \coref{4.12}, there exists a monic irreducible polynomial $f(y)\in k[y]$
such that $M=M^{(f)}$. Then $f(y)=y$ or $f(y)\neq y$.

Case 1: $f(y)=y$. Since $M$ is finite dimensional, $x^{rs}\cdot M=0$ for some $r\>1$.
Let $V$ be a composition factor of $M$. Then $x^{rs}\cdot V=0$. Obviously, $x\cdot V$
is a submodule of $V$, and hence $x\cdot V=0$ since $V$ is simple and $x^{rs}\cdot V=0$.
This shows that $V$ is one dimensional. It follows from \thref{4.8} that $M\cong V_t(\l)$ for some
$t\>1$ and $\l\in\widehat{G}$. By \prref{4.3}(b), $M$ is uniserial in this case.

Case 2: $f(y)\neq y$. Since $M=M^{(f)}$,
by \leref{4.15} and its proof, there exists a $\l\in\widehat{G}$
such that each composition factor of $M$ is isomorphic to $V(\l, f)$ and
$\Pi(M)=\{\chi^i\l|0\<i\<s-1\}$. If ${\rm rl}(M)=1$, then $M$ is simple, and so $M\cong V(\l, f)$.
Now assume ${\rm rl}(M)=r>1$. Define a linear map $\phi: M\ra M$
by $\phi(m)=f(x^s)\cdot m$ for all $m\in M$. Then $\phi$ is a module endomorphism of $M$
since $f(x^s)$ is central in $H$. It follows from \leref{4.13}, its proof and \leref{4.14}
that ${\rm Ker}(\phi)={\rm soc}(M)$. By \leref{4.16}, $\phi(N)={\rm rad}(N)$ for
any submodule $N$ of $M$. It follows from \leref{4.1}(c) that $M$ is uniserial,
and hence $M/{\rm rad}(M)\cong V(\l, f)$. Moreover, $f(x^s)^r\cdot M=\phi^r(M)={\rm rad}^r(M)=0$.
Hence there is an $H$-module epimorphism $\varphi: M(\l)\ra M/{\rm rad}(M)$.
It follows from \leref{4.6} that there is
an $H$-module map $\psi: M(\l)\ra M$ such that $\pi\psi=\varphi$,
where $\pi: M\ra M/{\rm rad}(M)$ is the canonical $H$-module epimorphism.
Since $\varphi$ is surjective, $M={\rm Im}(\psi)+{\rm rad}(M)$.
Hence $M={\rm Im}(\psi)$, and so $\psi$ is surjective.
Now we have $\psi(N_{f^r}(\l))=\psi(f(x^s)^r\cdot M(\l))=f(x^s)^r\cdot\psi(M(\l))=f(x^s)^r\cdot M=0$.
Hence $\psi$ induces an $H$-module epimorphism $\ol{\psi}: V(\l, f^r)\ra M$.
By \prref{4.17} and its proof, $V(\l, f^r)$ is indecomposable with ${\rm rl}(V(\l, f^r))=r$,
and $V(\l, f^r)=V(\l, f^r)^{(f)}$. Hence from the above discussion, one knows that
$V(\l, f^r)$ is uniserial. Hence ${\rm l}(V(\l, f^r))={\rm rl}(V(\l, f^r))=r={\rm rl}(M)={\rm l}(M)$,
which implies that $\ol{\psi}$ is an $H$-module isomorphism.
\end{proof}

\begin{corollary}\colabel{4.19}
Assume that $|\chi|=s<\oo$. Then
$$\left\{V_t(\l), V(\s, f^t)\left| \l\in\widehat{G}, [\s]\in\widehat{G}/\langle\chi\rangle,
t\>1, f(y)\in{\rm Irr}[y]\right.\right\}$$
is a representative set of isomorphic classes of finite dimensional
indecomposable objects in $\mathcal W$.
\end{corollary}
\begin{proof}
For $f_1(y), f_2(y)\in{\rm Irr}[y]$ and integers $t, r\>1$, one knows that $f_1(y)^t=f_2(y)^r$
if and only if $f_1(y)=f_2(y)$ and $t=r$. Hence the corollary follows from \leref{3.7},
Propositions 4.3, 4.4 and 4.17, 
and \thref{4.18}.
\end{proof}

When $f(y)=(y-\b)^n$ with $n\>1$ and $\b\in k$, we denote $V(\l, f)$ by $V_n(\l, \b)$,
where $\l\in\widehat{G}$. If $k$ is algebraically closed, then each monic irreducible polynomial in $k[y]$
has the form $y-\b$, $\b\in k$. Thus, from \coref{4.19}, we have the following corollary.

\begin{corollary}\colabel{4.20}
Assume that $|\chi|=s<\oo$ and $k$ is algebraically closed. Then
$$\left\{V_t(\l), V_t(\s, \b)\left| \l\in\widehat{G}, [\s]\in\widehat{G}/\langle\chi\rangle,
t\>1, \b\in k^{\times}\right.\right\}$$
is a representative set of isomorphic classes of finite dimensional
indecomposable objects in $\mathcal W$.
\end{corollary}

Now we consider the finite dimensional indecomposable weight modules over $H'$.
Note that $H'=H/I$, where $I$ is the Hopf ideal of $H$ as stated in \thref{2.7}.
From the paragraph before \coref{3.15}, if $\chi^{-1}(a)$ is not a primitive $n$-th root of unity for any
$n\>2$, then $I=0$ and $|\chi|=\infty$. In this case, $H'=H$, and the finite dimensional
indecomposable weight modules over $H'$ are exactly those over $H$ as displayed in \coref{4.9}.
If $\chi^{-1}(a)$ is a primitive $n$-th root of unity for some
$n\>2$, then $I=\langle x^n\rangle$,
or $I=\langle x^n-\b(1-a^n)\rangle$ with $\b\neq 0$, $a^n\neq1$ and $|\chi|=n$.

Now assume that $\chi^{-1}(a)$ is a primitive $n$-th root of unity,
where $n\>2$. Then $n\<|\chi|$ by \prref{2.8}(a).

For any $\l\in \widehat{G}$, let $\overline{M(\l)}:=M(\l)/(I\cdot M(\l))$.
Then $\overline{M(\l)}$ is a weight $H'$-module. For $m\in M(\l)$, let $\overline{m}$
denote the image of $m$ under the canonical epimorphism $\pi: M(\l)\ra\overline{M(\l)}$.

Case 1: $I=\langle x^n\rangle$. In this case, $\overline{M(\l)}\cong V_n(\l)$ for any $\l\in\widehat{G}$.

Case 2: $I=\langle x^n-\b(1-a^n)\rangle$ with $\b\in k^{\times}$, $a^n\neq 1$ and $n=|\chi|$.
In this case, one can easily check that $\ol{M(\l)}\cong V(\l, \b(1-\l(a)^n))$.
If $\l(a)^n\neq 1$ then $\ol{M(\l)}$ is simple; if $\l(a)^n=1$ then $\ol{M(\l)}\cong V_n(\l)$.
It follows from \prref{4.3}(b) that $\overline{M(\l)}$ is uniserial and indecomposable.
Moreover, for $\l, \s\in\widehat{G}$, $\ol{M(\l)}\cong \ol{M(\s)}$ if and only if either $\l=\s$,
or $[\l]=[\s]$ and $\l(a)^n\neq 1$ ($\s(a)^n\neq 1$).

\begin{corollary}\colabel{4.21}
Let $\l\in\widehat{G}$. Then $\ol{M(\l)}$ is an
indecomposable projective object in $\mathcal W'$.
\end{corollary}
\begin{proof}
Let $\l\in\widehat{G}$. Then $\overline{M(\l)}$ is an indecomposable
weight $H'$-module by the above discussion. Now let $f: M\ra L$ be an epimorphism
in $\mathcal W'$, and $g: \ol{M(\l)}\ra L$ a morphism in $\mathcal W'$.
Then $f$ and $g$ can be regarded as morphisms in $\mathcal W$. Hence by \leref{4.6},
there exists an $H$-module map $\phi: M(\l)\ra M$ such that $f\phi=g\pi$. Since $I\cdot M=0$,
we have $\phi(I\cdot M(\l))=I\cdot\phi(M(\l))=0$. Hence there is a unique $H$-module map
$\psi: \ol{M(\l)}\ra M$ such that $\psi\pi=\phi$. Thus, $f\psi\pi=f\phi=g\pi$, and so $f\psi=g$
since $\pi$ is surjective. Note that $\psi$ is also a morphism in $\mathcal W'$.
This completes the proof.
\end{proof}

\begin{proposition}\prlabel{4.22}
Assume that $\chi^{-1}(a)$ is a primitive $n$-th root of unity,
where $n\>2$.\\
{\rm (a)} If $I=\langle x^n\rangle$, then $\{V_t(\l)|\l\in\widehat{G}, 1\<t\<n\}$ is a representative set of
isomorphic classes of finite dimensional indecomposable objects in $\mathcal W'$.
Moreover, $V_n(\l)$ is a projective object in $\mathcal W'$ for all $\l\in\widehat{G}$.\\
{\rm (b)} If $I=\langle x^n-\b(1-a^n)\rangle$ with $\b\in k^{\times}$, $a^n\neq 1$ and $n=|\chi|$,
then
$$\left\{V_t(\l), V(\s, \b(1-\s(a)^n))\left|
\begin{array}{l}
1\<t\<n, \l\in\widehat{G} \mbox{ with }\l(a)^n=1,\\
{[\s]}\in \widehat{G}/\langle\chi\rangle \mbox{ with }\s(a)^n\neq 1\\
\end{array}
\right.\right\}$$
is a representative set of isomorphic classes of finite dimensional indecomposable
objects in $\mathcal W'$. Moreover, $V(\l, \b(1-\l(a)^n))$ is a projective object
in $\mathcal W'$ for any $\l\in\widehat{G}$.
\end{proposition}
\begin{proof}
(a) For any $\l\in\widehat{G}$ and $t\>1$, $x^n\cdot V_t(\l)=0$ if and only if $n\>t$.
Hence by \prref{4.3}, $V_t(\l)$ is an indecomposable $H'$-module for any $\l\in\widehat{G}$ and $1\<t\<n$.
Conversely, let $M$ be a finite dimensional indecomposable weight $H'$-module.
Then $M$ is an indecomposable weight $H$-module and $x^n\cdot M=0$.
It follows from the proof of \thref{4.18} that $M\cong V_t(\l)$ for some $\l\in\widehat{G}$
and $t\>1$. Then by the above discussion, $t\<n$. Thus, the first statement follows
from \prref{4.4}. The second statement follows from \coref{4.21} and the discussion
before \coref{4.21}.

(b) For any $\l\in\widehat{G}$ and $t\>1$, one can easily check that
$(x^n-\b(1-a^n))\cdot V_t(\l)=0$
if and only if $t\<n$ and $\l(a)^n=1$.
Hence by \prref{4.3}, $V_t(\l)$ is an indecomposable $H'$-module for any $1\<t\<n$ and
$\l\in\widehat{G}$ with $\l(a)^n=1$. By \coref{3.16}(b),
$V(\s, \b(1-\s(a)^n))$ is an indecomposable weight $H'$-module for any
$[\s]\in \widehat{G}/\langle\chi\rangle$  with $\s(a)^n\neq 1$.
Conversely, let $M$ be a finite dimensional indecomposable weight $H'$-module.
Then $M$ is an indecomposable weight $H$-module and $(x^n-\b(1-a^n))\cdot M=0$.
By \thref{4.18}, $M\cong V_t(\l)$ for some $\l\in\widehat{G}$ and $t\>1$, or
$M\cong V(\s, f^r)$ for some $[\s]\in \widehat{G}/\langle\chi\rangle$, $f(y)\in{\rm Irr}[y]$
and $r\>1$. If $M\cong V_t(\l)$ then $t\<n$ and $\l(a)^n=1$ as shown above.
Now assume that $M\cong V(\s, f^r)$. Then by \leref{4.15}, each composition factor of $M$
is isomorphic to $V(\s, f)$ since $\Pi(M)=\Pi(V(\s, f^r))=\{\chi^i\s|0\<i\<n-1\}$.
From the proof of \prref{3.6}, one also knows that $x\cdot m\neq0$ for any $0\neq m\in M$.
Hence $x^n\cdot M\neq 0$.
For any $v\in V(\s, f^r)_{(\chi^i\s)}$, we have
$a^n\cdot v=\chi(a)^{in}\s(a)^nv=\s(a)^nv$, where $0\<i\<n-1$.
Hence $(x^n-\b(1-a^n))\cdot v=(x^n-\b(1-\s(a)^n))\cdot v$ for all $v\in V(\s,f^r)$.
Thus, $(x^n-\b(1-\s(a)^n))\cdot M=0$. If $\s(a)^n=1$, then $x^n\cdot M=0$,
which is impossible. Hence $\s(a)^n\neq 1$, and then $f_1(y)=y-\b(1-\s(a)^n)\in{\rm Irr}[y]$, and
$f_1(x^n)\cdot M=0$. By \leref{4.14} and its proof, $M\cong V(\s, f_1)\cong V(\s, \b(1-\s(a)^n))$.
Thus, the first statement follows from the proof of \coref{3.16}(b) and \prref{4.4}.
The second statement follows from \coref{4.21} and the discussion before \coref{4.21}.
\end{proof}

For any $\l\in\widehat{G}$ and $1\<t<n$, there exists a module epimorphism $V_n(\l)\ra V_t(\l)$,
which is not split. Thus, by \prref{4.22} and its proof, we have the following corollary.

\begin{corollary}\lelabel{4.23}
Assume that $\chi^{-1}(a)$ is a primitive $n$-th root of unity,
where $n\>2$.\\
{\rm (a)} If $I=\langle x^n\rangle$, then $\{V_n(\l)|\l\in\widehat{G}\}$ is a representative set of
isomorphic classes of finite dimensional indecomposable projective objects in $\mathcal W'$.\\
{\rm (b)} If $I=\langle x^n-\b(1-a^n)\rangle$ with $\b\in k^{\times}$, $a^n\neq 1$ and $n=|\chi|$,
then
$$\left\{V_n(\l), V(\s, \b(1-\s(a)^n))\left|\l\in\widehat{G} \text{ with }
\l(a)^n=1, {[\s]}\in \widehat{G}/\langle\chi\rangle \text{ with }\s(a)^n\neq 1\right.\right\}$$
is a representative set of isomorphic classes of finite dimensional indecomposable projective
objects in $\mathcal W'$.
\end{corollary}

\begin{remark}\relabel{4.24}
Assume that $\chi^{-1}(a)$ is a primitive $n$-th root of unity,
where $n\>2$. Then from the above discussion, the projective covers of simple objects
in $\mathcal W'$  can be described as follows: If $I=\langle x^n\rangle$,
then $V_n(\l)$ is a projective cover of $V_{\l}$ in $\mathcal W'$ for any $\l\in\widehat{G}$.
If $I=\langle x^n-\b(1-a^n)\rangle$ with $\b\in k^{\times}$, $a^n\neq 1$ and $n=|\chi|$,
then $V_n(\l)$ is a projective cover of $V_{\l}$ in $\mathcal W'$ for any $\l\in\widehat{G}$ with $\l(a)^n=1$,
and $V(\s, \b(1-\s(a)^n))$ is its own projective cover in $\mathcal W'$
for any $[\s]\in \widehat{G}/\langle\chi\rangle$ with $\s(a)^n\neq 1$.
\end{remark}

\section*{Acknowledgments}
The authors thank the referee for the valuable comments.
This work is supported by NSF of China, No.11171291, and Doctorate united foundation,
No. 20123250110005, of Ministry of China and Jiangsu Province, and also supported by
Yancheng Institute of Technology Research Program, No. XKR2011022.

\end{document}